\documentclass{amsproc}
\usepackage{euscript}
\usepackage{cases}
\usepackage{mathrsfs}
\usepackage{bbm}
\usepackage{amssymb}
\usepackage{amsfonts,amsmath,amsxtra,mathdots,mathabx}
\usepackage{color}
\usepackage{hyperref}
\usepackage{tikz}
\usepackage{appendix,upgreek}

\allowdisplaybreaks

\DeclareFontFamily{U}{matha}{\hyphenchar\font45}
\DeclareFontShape{U}{matha}{m}{n}{
	<5> <6> <7> <8> <9> <10> gen * matha
	<10.95> matha10 <12> <14.4> <17.28> <20.74> <24.88> matha12
}{}
\DeclareSymbolFont{matha}{U}{matha}{m}{n}

\DeclareMathSymbol{\Lt}{3}{matha}{"CE}
\DeclareMathSymbol{\Gt}{3}{matha}{"CF}

\DeclareSymbolFont{mathc}{OML}{txmi}{m}{it}
\DeclareMathSymbol{\varuu}{\mathord}{mathc}{117}
\DeclareMathSymbol{\varvv}{\mathord}{mathc}{118}
\DeclareMathSymbol{\varww}{\mathord}{mathc}{119}

\def\ssstyle{\scriptscriptstyle}

\def\SB{\text{\raisebox{- 2 \depth}{\scalebox{1.1}{$ \text{\usefont{U}{BOONDOX-calo}{m}{n}B}   $}}}}

\def\SD{\text{\raisebox{- 2 \depth}{\scalebox{1.1}{$ \text{\usefont{U}{BOONDOX-calo}{m}{n}D} \hspace{0.5pt} $}}}}

\def\SE{\text{\raisebox{- 2 \depth}{\scalebox{1.1}{$ \text{\usefont{U}{BOONDOX-calo}{m}{n}E} \hspace{0.5pt} $}}}}

\def\SB{\text{\raisebox{- 2 \depth}{\scalebox{1.1}{$ \text{\usefont{U}{BOONDOX-calo}{m}{n}B} \hspace{0.5pt} $}}}}

\def\SM{\text{\raisebox{- 2 \depth}{\scalebox{1.1}{$ \text{\usefont{U}{BOONDOX-calo}{m}{n}M} \hspace{0.5pt} $}}}}

\def\SO{\text{\raisebox{- 2 \depth}{\scalebox{1.1}{$ \text{\usefont{U}{BOONDOX-calo}{m}{n}O} \hspace{0.5pt} $}}}}

\def\SP{\text{\raisebox{- 2 \depth}{\scalebox{1.1}{$ \text{\usefont{U}{BOONDOX-calo}{m}{n}P} \hspace{0.5pt} $}}}}

\def\SC{\text{\raisebox{- 2 \depth}{\scalebox{1.1}{$ \text{\usefont{U}{BOONDOX-calo}{m}{n}C}$}}}}

\def\valpha{\text{\scalebox{0.86}[1.02]{$\alpha$}}}   
\def\vepsilon{\upvarepsilon}
\def\vnu{\text{{\scalebox{0.86}[1]{$\nu$}}}} 
\def\vkappa{\text{{\scalebox{0.86}[1.1]{$\kappa$}}}} 
\def\vchi{\text{\scalebox{0.9}[1.06]{$\chi$}}}

\def\vQ  {\text{\scalebox{0.9}[1]{$Q$}}}

\newcommand{\BZ}{{\mathbf {Z}}}

\newcommand{\ra}{\rightarrow} 
\def\sumx{\sideset{}{^\star}\sum}
\def\sumd{\sideset{}{^{\delta}}\sum}

\def\mod{\mathrm{mod}\,  }
\def\nd{\mathrm{d}}

\def\lp {\left (}
\def\rp {\right )}

\def\shskip{\hspace{0.5pt}}

\newcommand{\delete}[1]{}

\theoremstyle{plain}

\newtheorem{thm}{Theorem} \newtheorem{cor}[thm]{Corollary}
\newtheorem{coro}{Corollary}[section]
\newtheorem{lem}{Lemma}[section]
\newtheorem{theorem}{Theorem}[section] 

\newtheorem*{thm*}{Theorem}

\theoremstyle{remark} 
\newtheorem{remark}{Remark}[section] 
\newtheorem{defn}{Definition}[section]

\newtheorem*{acknowledgement}{Acknowledgements}

\numberwithin{equation}{section}

\begin{document}
	

	\title[Non-vanishing of Hecke--Maass $L$-functions]{{On the Effective Non-vanishing of Hecke--Maass $L$-functions at Special Points}}    
	
	
	\begin{abstract}
	 In this paper, we consider the non-vanishing problem for the family of special Hecke--Maass $L$-values $ L (1/2+it_f, f) $ with $f (z)$ in an orthonormal basis of (even or odd) Hecke--Maass cusp forms of Laplace eigenvalue $1/4 + t_f^2$ ($t_f > 0$). We prove that 20\% of $L (1/2+it_f, f)$  for $ t_f \leqslant T$ do not vanish as $T \rightarrow \infty$. For comparison, it is known that the non-vanishing proportion is at least 25\% for the central  $L$-values $L (1/2, f)$. Further, 20\% may be raised to 50\% conditionally on the generalized Riemann hypothesis. Moreover, we prove non-vanishing results on the short interval $|t_f-T| \leqslant T^{\mu}$ for any $0 < \mu < 1$. 
	 
	
	\end{abstract}
	
	\author{Zhi Qi}
	\address{School of Mathematical Sciences\\ Zhejiang University\\Hangzhou, 310027\\China}
	\email{zhi.qi@zju.edu.cn}
	
	\thanks{The author was supported by National Key R\&D Program of China No. 2022YFA1005300.}

	
	\maketitle
	
	\vspace{-20pt}
	
	{\small \tableofcontents}

	\section{Introduction}
	
	

\subsection{Definitions} \label{defns} 
Let $ \SB $ be an orthonormal basis of Hecke--Maass  cusp forms  on the modular surface $\mathrm{SL}_2  ( \BZ) \backslash \mathbf{H} $.  Assume that each $f \in \SB $ is either even or odd  in the sense that $ f (- \widebar{z}) = (-1)^{\delta_f} f  (z)$ for $\delta_f = 0$ or $ 1$.   Let  $ \SB_0 $ or $\SB_1 $ be the subset of even or odd forms in $ \SB $ respectively. For $f \in \SB $, let   $\lambda_{f} = s_{f} (1-s_{f})$   
be its Laplace eigenvalue, 
with $s_f = 1/2+ i t_f$ ($t_f > 0$). 
The Fourier expansion of $f (z)$ reads:
\begin{align*}
	f (x+iy) =   \sqrt{y} \sum_{n \neq 0}  \rho_f (n) K_{i t_f} (2\pi |n| y) e (n x), 
\end{align*}
where as usual $K_{\vnu} (x)$ is the $K$-Bessel function and $e (x) = \exp (2\pi i x)$.  Let $\lambda_f (n)$ ($n \geqslant 1$) be the Hecke eigenvalues of $f (z)$. It is well known that  $\rho_f (\pm n) =   \allowbreak \rho_f (\pm 1)  \lambda_f (n) $, while $\rho_f (-1) = (-1)^{\delta_f} \rho_f (1)$.  
Define the harmonic weight 
\begin{align*}
	\omega_f = \frac {|\rho_f (1)|^2} {\cosh \pi t_f} .    
\end{align*}
The Hecke--Maass $L$-function of $f (z)$ is defined by
\begin{align*}
L (s, f) = \sum_{n    =1}^{\infty}\frac{ \lambda_f (n)  }{n^s} ,
\end{align*}
for $\mathrm{Re} (s) > 1$, and by analytic continuation on the whole complex plane.

\subsection{Effective Non-vanishing Results}

As the motivation of this paper,  the author \cite{Qi-GL(2)xG(2)-RS} recently established an effective refinement (recorded below) of the  non-vanishing result of Luo \cite{Luo-Weyl}  for the Rankin--Selberg special $L$-values  $L (s_f, \vQ \otimes f)$. The non-vanishing condition for  $L (s_f, \vQ \otimes f)$   arises in the Phillips–Sarnak deformation theory of Maass cusp forms \cite{Phillips-Sarnak} and was studied by Deshouillers, Iwaniec, and Luo \cite{DI-Nonvanishing,Luo-Non-Vanishing,Luo-Weyl,Luo-2nd-Moment}. 

\begin{thm*}[{\rm\cite{Qi-GL(2)xG(2)-RS}}]
	Let $\vQ (z)$ be a fixed holomorphic cusp newform of even weight $2k$ and square-free level $q$. 
	Let $3/4 < \mu < 1$. We have 
	\begin{align} \label{1eq: main, RS, 0}
	\liminf_{T \ra \infty}	\frac {	\text{\rm \small \bf \#} \big\{ f \in \SB   :    t_f   \leqslant   T , \, L (s_f , \vQ \otimes  f    ) \neq 0
			\big\}} { \text{\rm \small \bf \#} \big\{ f \in \SB   :    t_f   \leqslant   T 
			\big\} } \geqslant     \frac { \gamma (\vQ)} {11}      ,  
	\end{align} 
	\begin{align} \label{1eq: main, RS}
	\liminf_{T \ra \infty}	\frac {	\text{\rm \small \bf \#} \big\{ f \in \SB   :   |t_f - T| \leqslant   T^{\mu}, \, L (s_f , \vQ \otimes  f    ) \neq 0
			\big\}} { \text{\rm \small \bf \#} \big\{ f \in \SB   :   |t_f - T| \leqslant   T^{\mu} 
			\big\} } \geqslant  \gamma (\vQ)    \frac {4\mu-3} {4\mu+7}      ,  
	\end{align} 
where $\gamma (\vQ)$ is the Euler product\shskip {\rm:}
	\begin{align}\label{1eq: beta(Q)}
		\begin{aligned}
			\prod_{p \nmid q}  \bigg(1 - \frac {a(p)^2} {p(p+1) } \bigg)^{-1}  \bigg(1 + \frac {a(p)^2} {p+1 } \bigg)  \bigg( 1 - \frac {a(p)^2 - 2} {p} + \frac 1 {p^2}  \bigg)  \bigg( 1 - \frac {1} { p } \bigg)^2  \cdot  \prod_{p | q} \bigg( 1 - \frac 1 {p^2 } \bigg) ,
		\end{aligned}
	\end{align} 
	with $a (p)$ the $p$-th Hecke eigenvalue of $\vQ (z)$. 
\end{thm*}

It is natural to consider the more fundamental family of Hecke--Maass $L$-functions $L (s, f)$ and to see how large a non-vanishing proportion  at $s = s_f$ one might attain.  
It is also interesting to compare our non-vanishing results for $L (s_f, f)$ with those for central $L$-values $L(1/2, f)$, with or without the Riemann hypothesis.  

To be explicit, our main theorems for the non-vanishing statistic of the special $L$-values $L (s_f, f)$ are as follows. 

\begin{thm}
	\label{thm: non-vanishing}
	 We have 
	\begin{align} \label{1eq: main, long}
		\liminf_{T \ra \infty}		\frac {	\text{\rm \small \bf \#} \big\{ f \in \SB_{\delta}  :    t_f   \leqslant   T , \, L (s_f ,   f    ) \neq 0
			\big\}} { \text{\rm \small \bf \#} \big\{ f \in \SB_{\delta}  :  t_f \leqslant   T 
			\big\} } \geqslant    \frac {2} {\pi^2}  
	\end{align} 
for  $\delta = 0, 1${\rm;} 
 {\rm\eqref{1eq: main, long}} is also valid for the full family $\SB$. 
\end{thm}

 On the long interval $t_f \leqslant T$, Theorem \ref{thm: non-vanishing}  manifests that at least $2 / \pi^2 \approx 20 \%$ of the special $L$-values $L (s_f, f) $ do not vanish for $f \in \SB$, while it is proven in \cite{BHS-Maass,Qi-Liu-Moments} that 25\% of the central $L$-values $L(1/2, f)$ do not vanish for $f \in \SB$, although, in the latter case, $L (1/2, f) = 0$ trivially on the whole odd sub-basis $\SB_1$  as the root number $\epsilon_f = - 1$. 

\begin{thm}\label{thm: non-vanishing, short} 
Let $ 0  < \mu < 1$. 
We have 
	\begin{align} \label{1eq: main}
\liminf_{T \ra \infty}		\frac {	\text{\rm \small \bf \#} \big\{ f \in \SB_{\delta}  :   |t_f - T| \leqslant   T^{\mu}, \, L (s_f ,   f    ) \neq 0
		\big\}} { \text{\rm \small \bf \#} \big\{ f \in \SB_{\delta}  :   |t_f - T| \leqslant   T^{\mu} 
		\big\} } \geqslant  \frac {6} {\pi^2}  \min \bigg\{ \frac 1 3, \frac { \mu  } { \mu + 1 }  \bigg\}   .
	\end{align} 
\end{thm}

This is an effective refinement of the result of Xu \cite{Xu-Nonvanishing} in which he proved that (on the even  sub-basis $\SB_0$) the ratio in \eqref{1eq: main} is bounded below by an unexplicit  constant  as $T \ra \infty$.

\begin{thm}
	\label{thm: non-vanishing, 3}
	Assume the Riemann hypothesis for every $L (s, f)$ with $ f \in \SB_{\delta} $ and for all Dirichlet $L$-functions  {\rm(}including the Riemann  zeta function{\rm)}. Then 
	\begin{align} \label{1eq: main, long, RH}
	\liminf_{T \ra \infty}		\frac {	\text{\rm \small \bf \#} \big\{ f \in \SB_{\delta}  :    t_f   \leqslant   T , \, L (s_f ,   f    ) \neq 0
			\big\}} { \text{\rm \small \bf \#} \big\{ f \in \SB_{\delta}  :  t_f \leqslant   T 
			\big\} } \geqslant    \frac 1 2    .
	\end{align}  
\end{thm}

Thus, on  the Riemann hypothesis, for $f \in \SB$ the non-vanishing percentage 20\% for $ L (s_f, f) $ is  raised drastically to 50\%, but 25\% for $L(1/2, f)$ is  raised only to 28\% (see Corollary \ref{cor: long}).  

\begin{thm}
	\label{thm: non-vanishing, 4}
Assume the Riemann hypothesis as in Theorem \ref{thm: non-vanishing, 3}.	 Then 
		\begin{align} \label{1eq: main, RH}
	\liminf_{T \ra \infty}		\frac {	\text{\rm \small \bf \#} \big\{ f \in \SB_{\delta}  :   |t_f - T| \leqslant   T^{\mu}, \, L (s_f ,   f    ) \neq 0
			\big\}} { \text{\rm \small \bf \#} \big\{ f \in \SB_{\delta}  :   |t_f - T| \leqslant   T^{\mu} 
			\big\} } \geqslant  \min \bigg\{ \frac {\mu} {\mu+1}, \frac {3\mu-1} {3\mu}   \bigg\}    .
	\end{align} 
\end{thm}

These conditional results will be deduced from the extended density theorem for $L (s, f)$ at $s = s_f$.

\delete{It is curious that for $1/3 < \mu \leqslant 1/2$, the method only yields 
\begin{align}
\liminf_{T \ra \infty}		 \frac {	\text{\rm \small \bf \#} \big\{ f \in \SB_{\delta}  :   |t_f - T| \leqslant   T^{\mu}, \, L (s_f ,   f    ) \neq 0
	 	\big\}} { \text{\rm \small \bf \#} \big\{ f \in \SB_{\delta}  :   |t_f - T| \leqslant   T^{\mu} 
	 	\big\} } \geqslant  \frac {3\mu-1} {3\mu}     , 
\end{align}
but this is weaker than the unconditional  \eqref{1eq: main} in Theorem \ref{thm: non-vanishing, short}. }  


\begin{center}
\begin{figure}
		\begin{tikzpicture}[x=6.5cm,y=6.5cm,>=stealth,
		projection/.style={thin,dash pattern=on 3pt off 3pt}
		]
		\draw[->] (0,0) -- (1.06,0) node[right] {\footnotesize $\mu$};
		\draw[->] (0,0) -- (0,0.55);
		
		\foreach \x/\lab in {
			0/{\footnotesize $ 0$ },
			0.3333333333/{\footnotesize$\displaystyle \frac13$},
			0.5/{\footnotesize $\displaystyle  \frac12$},
			1/{\footnotesize$1$}
		}{
			\draw (\x,0) -- (\x,0.008);
			\node[below] at (\x,-0.008) {\lab};
		}
		
		\foreach \y/\lab in {
				0.2026423673/{\footnotesize$\displaystyle \frac{2}{\pi^2}\! $},
			0.3333333333/{\footnotesize$\displaystyle \frac13$},
			0.5/{\footnotesize $\displaystyle \frac12$}
		}{
			\draw (0,\y) -- (0.008,\y);
			\node[left] at (-0.008,\y) {\lab};
		}
		
		\draw[projection] (0.5,0) -- (0.5,0.3333333333) -- (0,0.3333333333);
		\draw[projection] (1,0) -- (1,0.5) -- (0,0.5);
		\draw[projection] (1,0) -- (1,{2/(pi*pi)}) -- (0,{2/(pi*pi)});
		
		\draw[gray, very thick, domain=0.3333333333:0.5, samples=100]
		plot ({\x},{(3*\x-1)/(3*\x)});
		\draw[gray, very thick, domain=0.5:1, samples=100]
		plot ({\x},{\x/(1+\x)});
		
		\draw[black, very thick, domain=0:0.5, samples=100]
		plot ({\x},{6/(pi*pi)*\x/(1+\x)});
		\draw[black, very thick]
		(0.5,{2/(pi*pi)}) -- (1,{2/(pi*pi)});
		
		\fill[gray] (0.3333333333,0) circle (0.008);
		\fill[white] (0.3333333333,0) circle (0.003);
		\fill[gray] (0.5,0.3333333333) circle (0.008);
		\fill[white] (0.5,0.3333333333) circle (0.003);
		\fill[gray] (1,0.5) circle (0.008);
		\fill[white] (1,0.5) circle (0.003);
		\fill[black] (0.5,{2/(pi*pi)}) circle (0.008);
		\fill[white] (0.5,{2/(pi*pi)}) circle (0.003);
		\fill[black] (1,{2/(pi*pi)}) circle (0.008);
		\fill[white] (1,{2/(pi*pi)}) circle (0.003);
		\fill[black] (0.4040075548,0.1749329205) circle (0.008);
		\fill[white] (0.4040075548,0.1749329205) circle (0.003);
		\fill[black] (0,0) circle (0.008);
		\fill[white] (0,0) circle (0.003);
	\end{tikzpicture}
\caption{A comparison of the proportions in \eqref{1eq: main} and \eqref{1eq: main, RH}.}
\label{fig:comparison}
\end{figure}
\end{center}

\subsection{Twisted Spectral Moments} 
Our proof of Theorems \ref{thm: non-vanishing} and  \ref{thm: non-vanishing, short} will rely on the asymptotic formulae for the twisted first and second moments of $ L (s_f ,   f    ) $.  

Let $T, \varPi$ be large parameters such that $T^{\vepsilon} \leqslant \varPi \leqslant T^{1-\vepsilon}$.  Define the smoothly weighted twisted moments on the (even or odd) cuspidal spectrum: 
\begin{align}\label{2eq: moments M1}
	{\SC^{\delta}_{1} }    (m  ) = \sum_{f \in \SB_{\delta} } \omega_f \lambda_f (m) m^{-it_f}  L (s_f,    f     )    \exp \left(  - \frac {(t_f - T)^2}  {\varPi^2} \right) , 
\end{align}
\begin{align}\label{2eq: moments M2}
	\begin{aligned}
		\SC^{\delta}_{2}  (m_1, m_2) = \mathrm{Re}  \sum_{f \in \SB_{\delta} }    \omega_f   \frac {\lambda_f ( m_1  ) \lambda_f ( m_2  )  } {(m_1/m_2)^{it_f } }  |L (s_f,    f     )|^2     \exp \left(  - \frac {(t_f - T)^2}  {\varPi^2} \right)  . 
	\end{aligned}
\end{align}

\begin{thm}\label{thm: C1(m)} 
Let the notation be as above. 	
 We have  
	\begin{align}\label{2eq: C1(m)}
		\SC^{\delta}_{1}   (m  ) = \frac {\varPi T} {\pi\sqrt{\pi}}  \cdot {\delta ({m, 1})}  + O_{  \vepsilon} \big((\varPi +m ) T^{1/2+\vepsilon} \big), 
	\end{align}
where $\delta (m ,1)$ is the Kronecker $\delta$ that detects $m = 1$.  
\end{thm}

\begin{thm}\label{thm: C2(m)}  
Let the notation be as above. 	Assume that $m_1, m_2 < \varPi^{1-\vepsilon}$.   Set \begin{align}\label{1eq: m, r}
	m = (m_1, m_2), \qquad r = \frac {m_1m_2}  {m^2}. 
\end{align}  
 Then 
	\begin{align}\label{2eq: C2(m)}
		\begin{aligned}
			\SC_{2}^{\delta}  (m_1, m_2 ) = \frac { \varPi T  } {\pi \sqrt{\pi r} }  & \bigg(  \bigg(   \log  {\frac T r}   + \gamma_{\delta}   \bigg) \sigma_{-1} (m)  - 2 \sigma_{-1}' (m)  \bigg)   \\
			& + O_{\vepsilon} \bigg(  T^{\vepsilon} \bigg(   \varPi   \sqrt{T} +  \frac { T  } {\sqrt{ r      } } +   \frac {  \varPi^3  } {T \sqrt{ r      } }  \bigg) \bigg)    , 
		\end{aligned}
	\end{align}
where 
\begin{align}\label{1eq: gamma}
	\gamma_{0 }  =    \gamma   - \log 8 \pi^2 - \frac {\pi} 2 ,  \qquad \gamma_1 = \gamma   - \log 8 \pi^2 + \frac {\pi} 2, 
\end{align} 
with $\gamma$   the Euler constant, and 
\begin{align}\label{1eq: Sigma}
	\sigma_{-1}   ( m    ) =	 {\sum_{d  | m     }  }  
	\frac {1} {d }, \qquad  \sigma_{-1}'   ( m  ) =	 {\sum_{d  | m      }  }  
	\frac {\log d } {d } . 
\end{align} 
\end{thm}

\begin{cor}\label{cor: 1st moment}
Let  $T^{\vepsilon} \leqslant \varPi \leqslant T^{1-\vepsilon}$. 	
We have 
	\begin{align}
		\SC_{1}^{\delta} (1)  = \frac {\varPi T} {\pi\sqrt{\pi}}   + O_{  \vepsilon} \big(\varPi T^{1/2+\vepsilon} \big), 
	\end{align}   
	\begin{align}\label{2eq: C2(1)}
		\SC_{2}^{\delta}  (1,1) = \frac {  \varPi T } {\pi \sqrt{\pi  }}   (  \log   {T}   + \gamma_{\delta}   ) + O_{  \vepsilon} \bigg(T^{\vepsilon} \bigg(  {\varPi \sqrt{T}} + T  + \frac {\varPi^3} {T} \bigg)\bigg). 
	\end{align} 
\end{cor}

\begin{cor}\label{cor: unsmoothed}
	We have 
	\begin{align}\label{2eq: 1st moment}
		\sumd_{  t_f \leqslant T} \omega_f L (s_f,    f     ) = \frac {T^{2}} {2 \pi^2}  + O_{  \vepsilon} \big(T^{3/2+\vepsilon}\big), 
	\end{align} 
\begin{align}\label{2eq: 2nd moment}
	\sumd_{  t_f \leqslant T} \omega_f |L (s_f,    f     )|^2 = \frac {1 } {2 \pi^2} T^2 \log T  + \frac {2\gamma_{\delta} - 1 } {4 \pi^2 } T^2  + O_{  \vepsilon} \big(T^{3/2+\vepsilon}\big), 
\end{align}
where the superscript $\delta$   indicates   summation over $\SB_{\delta}$. 
\end{cor}

\subsection{Density Theorems} 

Our proof of Theorems \ref{thm: non-vanishing, 3} and \ref{thm: non-vanishing, 4} will rely on the extended density theorem under the Riemann hypothesis for Dirichlet $L$-functions. 


Let us denote the non-trivial zeros of  $L (s, f)$ by $$\rho_f = \frac 1 2 + i \gamma_f; $$ if    the generalized Riemann hypothesis for $L (s, f)$ is assumed then $\gamma_f$ are all real.

\begin{defn}\label{defn: density}
	Define the $1$-level density of $L (s, f)$ at $s_f = 1/2+ it_f$ to  be 
	\begin{align}\label{1eq: defn D(f)}
		D  (f; \upphi; T) = \sum_{\gamma_f} \upphi \bigg(\frac {\log T} {2\pi} (\gamma_f - t_f)  \bigg) ,
	\end{align}
where $T$ is a large scaling parameter and $\upphi : \mathbf{R} \ra \mathbf{C}$ is an even Schwartz function such that its Fourier transform
\begin{align*}
	\hat{\upphi} (y) = \int_{-\infty}^{\infty} \upphi (x) e (-xy) \nd x
\end{align*}
has compact support. 
\end{defn}

For $T, \varPi$ as above, define the smoothly averaged $1$-level density
\begin{align}\label{1eq: density}
	\SD^{\delta}   ( \upphi;  T, \varPi ) =   \sum_{f \in \SB_{\delta} } \omega_f	 D (f; \upphi; T)  \exp \left(  - \frac {(t_f - T)^2}  {\varPi^2} \right)  .
\end{align}

\begin{thm}\label{thm: density, limited}
	Let the notation be as above. Assume the Riemann hypothesis {\rm(}for the Riemann  zeta function{\rm)}. Then for $\hat{\upphi}$ of support in $(-1, 1)$, we have
	\begin{align}
		\lim_{T \ra \infty} \frac {\pi \sqrt{\pi}} {T^{1+\mu}}  \SD^{\delta}   ( \upphi;  T, T^{\mu} ) = \hat{\upphi} (0) =  \int_{-\infty}^{\infty} \upphi (x) W (\mathrm{U}) (x) \nd x ,
	\end{align}
 for any fixed $0 < \mu < 1$. 
\end{thm}

Recall that $W (\mathrm{U}) (x) \equiv  1$ according to Katz--Sarnak \cite{KS-2}.

Theorem \ref{thm: density, limited} manifests that the symmetry type of $ L (s_f , f) $ on $ \SB_0 $ and $\SB_1$ are both unitary, while it is proven in \cite[Appendix]{Qi-Liu-LLZ} that the symmetry types of $ L (1/2 , f) $ on  $ \SB_0 $ and $\SB_1$ are orthogonal of different parity: $\mathrm{SO} (\mathrm{even})$ and $\mathrm{SO} (\mathrm{odd})$, respectively. Note that for $\hat{\upphi}$ of support in $(-1, 1)$ one is unable to differentiate $\mathrm{SO} (\mathrm{even})$ and $\mathrm{SO} (\mathrm{odd})$. 

However, in order to obtain a non-trivial non-vanishing result for $ L (s_f , f) $ (as in Theorems \ref{thm: non-vanishing, 3} and \ref{thm: non-vanishing, 4}), we must let the support of $ \hat{\upphi} $ go beyond $[-1, 1]$ as in the following extended density theorem.  

\begin{thm}\label{thm: density} 
Define  $v (\mu) = \min \{ 3\mu, 1 + \mu \}$. Then  for $1/3 < \mu < 1$,  we have 
\begin{align}\label{1eq: lim, extended}
		\lim_{T \ra \infty} \frac {\pi \sqrt{\pi}} {T^{1+\mu}}  \SD^{\delta}   ( \upphi;  T, T^{\mu} ) = \hat{\upphi} (0) , 
\end{align} for $\hat{\upphi}$ of support in $(-v (\mu) , v(\mu) )$ if we assume in addition the Riemann hypothesis for Dirichlet $L$-functions. 
\end{thm}

\begin{remark}
	{We have the  smoothly weighted harmonic Weyl law (see \eqref{2eq: Weyl, 1}): 
\begin{align} \label{1eq: Weyl, smooth}
	\sum_{f \in \SB_{\delta} } \omega_f	\exp \left(  - \frac {(t_f - T)^2}  {\varPi^2} \right) = \frac {\varPi T } {\pi \sqrt{\pi}}  + O (T). 
\end{align}}
\end{remark}

\begin{remark}
	Note that in Theorems {\rm\ref{thm: density, limited}} and {\rm\ref{thm: density} } there is no need to assume  the Riemann hypothesis for $L(s, f)$, since the definition  in {\rm\eqref{1eq: defn D(f)}} makes sense even if $\gamma_f$ are not real as  $\upphi$ is entire.  
\end{remark}

\subsection*{Notation}

By $X \Lt Y$ or $X = O (Y)$ we mean that $|X| \leqslant c Y$  for some constant $c  > 0$, and by $X \asymp Y$ we mean that $X \Lt Y$ and $Y \Lt X$. We write $X \Lt_{\valpha, \beta, ...} Y $ or $  X = O_{\valpha, \beta, ...} (Y) $ if the implied constant $c$ depends on $\valpha$, $\beta$, ....  

By `negligibly small' we mean $ O_A ( T^{-A} )$ for arbitrarily large but fixed $A $.  

The notation $x \sim X$ stands for  $ X \leqslant  x \leqslant 2 X $.  


Throughout the paper,  $\vepsilon  $ is arbitrarily small and its value  may differ from one occurrence to another.

\begin{acknowledgement}
	The author gratefully acknowledges the referees for identifying several errors in the manuscript and offering numerous valuable suggestions that substantially improved the exposition.
\end{acknowledgement}

\section{Kuznetsov Trace Formulae} \label{sec: Kuz}

Let  $h (t)$ be an even function satisfying the conditions{\hspace{0.5pt}\rm:}
\begin{enumerate} 
	\item[{\rm (i)\,}] $h (t)$ is holomorphic in  $|\operatorname{Im}(t)|\leqslant {1}/{2}+\vepsilon$,
	\item[{\rm (ii)}] $h(t)\Lt (|t|+1)^{-2-\vepsilon}$ in the above strip. 
\end{enumerate}	For $m, n \geqslant 1$, define the cuspidal term
\begin{align}
	\Delta_{ \delta} (m, n) = \sum_{f \in \SB_{\delta}}  \omega_f h(t_f) \lambda_f ( m )  { \lambda_f ( n) } ,
\end{align}
and the Eisenstein term 
\begin{align}
	\Xi (m, n )	= \frac{1}{\pi}   \int_{-\infty}^{\infty} \omega(t) h(t)  \tau_{it}(m) \tau_{it} (n) \nd t, 
\end{align}
where $\omega_f$ and $\lambda_{f} (n)$ are defined before in \S \ref{defns}, and
\begin{align}\label{2eq: divisor}
	\omega (t) = \frac {1 } {|\zeta (1 + 2it)|^2} , \qquad 	\tau_{\vnu} (n) = 
	\sum_{ a b = n} (a/b)^{  \vnu} .  
\end{align}
On the other hand, define the Plancherel integral
\begin{align}\label{app: integral H}
	H   = \frac {1} {  \pi^2} \int_{-\infty}^{\infty} h(t)    {  \tanh (\pi t) t  \nd t }    , 
\end{align}
and the Kloosterman--Bessel sum 
\begin{align}\label{app: KB +-}
	\mathrm{KB}^{  \pm} (m, n ) =  \sum_{  c = 1 }^{\infty}  \frac {S (  m , \pm n ; c)} { c}  H^{  \pm} \bigg(    \frac {4\pi \textstyle \sqrt { m n } } { c} \bigg),
\end{align} 
where $S (m, \pm n; c)$ is the Kloosterman sum and  $H^{\pm} (x)$ are the Bessel integrals
\begin{align} \label{app: H + (x)}
	 & H^{  +} (x) = \frac {2 i} {\pi} \int_{-\infty}^{\infty}   h(t) J_{2it} ( x) \frac {t   \nd t  } {\cosh (\pi t)},   \\
	\label{app: H - (x)}
	&  H^{  -} (x) = \frac 4 {\pi^2  }   \int_{-\infty}^{\infty}   h(t) K_{2it} (  x)   {\sinh (\pi t)} t   \nd t  .
\end{align}
By \cite[3.7 (6)]{Watson}, $$2 K_{\vnu} (z) = \pi \frac {I_{-\vnu} (z) - I_{\vnu} (z)} {\sin (\pi \vnu )} ,$$ 
where $I_{\vnu} (z)$ is the $I$-Bessel function. Hence,  we may rewrite \eqref{app: H - (x)} in a form which is similar to \eqref{app: H + (x)}: 
\begin{align}\label{app: H - (x), 2}
	H^{  -} (x) = \frac {2 i} {\pi} \int_{-\infty}^{\infty}   h(t) I_{2it} (  x) \frac {t  \nd t  } {\cosh (\pi t)} . 
\end{align}
The Kuznetsov trace formula for the even or odd cusp forms reads (see \cite[\S 3]{CI-Cubic}): 
\begin{align}\label{2eq: Kuznetsov}
	2	\Delta_{  \delta }  (m, n )  + 2 (1-\delta) \shskip \Xi   (m, n) =    \delta (m, n) \shskip H  +  \mathrm{KB}^{  +}  (m, n) + (-1)^{\delta} \mathrm{KB}^{  -}  (m, n) ,
\end{align}
where   $\delta (m, n)$ is the Kronecker $\delta$-symbol. 

\subsection{The Kloosterman Sum} Recall that 
\begin{align}\label{3eq: Kloosterman}
	S   (m, n ; c ) = \sumx_{   \valpha      (\mathrm{mod} \, c) } e \bigg(   \frac {  \valpha     m +   \widebar{\valpha    } n} {c} \bigg) ,
\end{align} 
where the $\star$ indicates the condition $(\valpha    , c) = 1$ and $ \widebar{\valpha    }$ is given by $\valpha     \widebar{\valpha    }  \equiv 1 (\mathrm{mod} \, c)$. 
We have the Weil bound: 
\begin{align}
	\label{3eq: Weil}
	S (m, n; c)   \Lt \tau (c) \sqrt{ (m, n, c) } \sqrt{c} ,  
\end{align} 
where as usual $\tau (c) = \tau_0 (c) $ is the number of divisors of $c$. 

\subsection{The Hecke Relation}
We have 
\begin{align}\label{3eq: Hecke}
	\lambda_f (m) \lambda_f (n) = \sum_{d | (m, n)} \lambda_f (mn/d^2), \qquad \tau_{\vnu} (m)  \tau_{\vnu} (n) = \sum_{d | (m, n)}  \tau_{\vnu} (mn/d^2).  
\end{align}

\subsection{The Weyl Law} 

Finally, we record here the harmonic weighted Weyl law for $\SB_{\delta}$:
\begin{align}\label{2eq: Weyl, 1}
	\sum_{ f\in \SB_{\delta} : \, t_f \leqslant T} \omega_f = \frac {T^2} {2 \pi^2} + O (T). 
\end{align} 
By the method in \cite{XLi-Weyl}, this follows from the Kuznetsov trace formula for $\SB_{\delta}$.

\section{Hecke--Maass $L$-functions}   For $f \in \SB$, define 
\begin{align}\label{3eq: L(s,f)}
	L (s, f) = \sum_{n    =1}^{\infty}\frac{ \lambda_f (n)  }{n^s} ,
\end{align}
for $\mathrm{Re} (s) > 1$. It is known that $L (s, f)$ has analytic continuation to the whole complex plane and satisfies the functional equation
\begin{align}\label{3eq: Lambda (Q)}
	\Lambda (s, f) = \epsilon_f   \Lambda (1-s, f),   \qquad \epsilon_f = (-1)^{\delta_f} , 
\end{align} 
with
\begin{align}\label{3eq: FE, Q}
	\Lambda (s, f) = \Gamma_{\mathbf{R}}^{\delta_f} (s  - it_f) \Gamma_{\mathbf{R}}^{\delta_f} (s  + it_f) L (s, f) , 
\end{align}  
where
\begin{align}\label{3eq: GammaR(s)} 
	 \Gamma_{\mathbf{R}}^{\delta} (s) = \pi^{-s/2} \Gamma ((s+\delta) /2).  
\end{align}
Note that, in parallel to \eqref{3eq: L(s,f)}, if we replace $\lambda_f (n)$ by  $  \tau_{ it} (n) $ as defined in \eqref{2eq: divisor}, then  
\begin{align}\label{3eq: RS, Eis}
	\zeta (s -it   ) \zeta (s + it  ) =   \sum_{n    =1}^{\infty}\frac{   \tau_{it} (n    )}{n^{s }} .
\end{align} 
Moreover, we have the Euler factorization
\begin{align}\label{3eq: Euler}
	L (s, f) = \prod_{p} \bigg( 1 - \frac {\valpha_f (p)} {p^s} \bigg)^{-1} \bigg( 1 - \frac {\beta_f (p)} {p^s} \bigg)^{-1} , 
\end{align} 
for $\mathrm{Re} (s) > 1$, and the Kim--Sarnak bound \cite{Kim-Sarnak}
\begin{align}\label{3eq: Kim-Sarnk}
	|\valpha_f (p) |, \, |\beta_f (p) | \leqslant p^{\frac 7 {64}}. 
\end{align}

	\section{Analysis for the Bessel Integrals}\label{sec: Bessel}

	Let $T, \varPi$ be large parameters with $ T^{\vepsilon} \leqslant \varPi \leqslant T^{1-\vepsilon} $. 
	In this section, we consider the Bessel transforms of test functions in the form
	\begin{align}\label{4eq: defn h}
		h (t; y) = \upgamma (t) y^{-2it} \upvarphi (t) + \upgamma (- t) y^{2it} \upvarphi (- t), 
	\end{align} 
	for $\upgamma   (t) $ in the function space $ \mathscr{Q}_{\vkappa } $ or $ \mathscr{T}_{\tau } $ defined below in Definition  \ref{defn: space Q(r)} or \ref{defn: space T(r)},   and 
	\begin{align}\label{4eq: defn phi}
		\upvarphi (t) = \exp \bigg(  \!      -   \frac {(t - T)^2 } {\varPi^2 }    \bigg) . 
	\end{align}
In the case that $\upgamma (t)$ is even, we also need to consider
\begin{align}\label{4eq: h2(t;y)}
	h_2 (t; y) = h (t; y) + h (t; 1/y) = 2 \upgamma (t) \cos (2 t \log y) (\upvarphi (t) + \upvarphi (-t)) . 
\end{align}
Let $H^{\pm} (x, y)$ and $H^{\pm}_2 (x, y)$ denote the Bessel transforms of $h (t; y)$ and $ h_2 (t; y)  $ as in \eqref{app: H + (x)} and \eqref{app: H - (x)} or \eqref{app: H - (x), 2}. Note that 
\begin{align}\label{4eq: H2=H}
	H_2^{\pm} (x, y) = H^{\pm} (x, y) + H^{\pm} (x, 1/y), 
\end{align} 
so the results below for $H^{\pm} (x, y) $ will be valid for  $H^{\pm}_2 (x, y) $ as well.

		\begin{defn}\label{defn: space Q(r)}
		For fixed real $\vkappa $, let $ \mathscr{Q }_{\vkappa } $ denote the space of holomorphic functions $\upgamma   (t) $ on the strip $ |\mathrm{Im} (t)| \leqslant 1/2 + \vepsilon$ such that 
		\begin{align}\label{4eq: bound for w, 1}
			 \upgamma (t)    \Lt (|t|+1)^{\vkappa    }   .
		\end{align} 
	\end{defn}

\begin{defn}\label{defn: space T(r)}
	For fixed real $\tau $, let $ \mathscr{T }_{\tau } \subset C^{\infty} (\mathbf{R})$ denote the space of smooth functions on $ \mathbf{R} $ such that 
	\begin{align}\label{4eq: bound for u, 2}
		\upgamma^{(n)}   (t)    \Lt_{n } (|t|+1)^{ \tau }  \lp \frac { \log (|t|+2) } { |t|+1 } \rp^{n}   , 
	\end{align} 
	for any $n \geqslant 0$. 
	
\end{defn}

The Bessel integral $H^{+} (x, y)$ has been studied thoroughly in \cite[\S 4]{Qi-GL(2)xG(2)-RS}. The results therein may be extended to $ H^{-} (x, y) $ by literally the same proofs.

\begin{lem}\label{lem: x<1}
Let $\upgamma (t) $ be in the space $ \mathscr{Q}_{\vkappa}$. Set $u = xy +x/y$. Then for $u \Lt 1  $,  we have   $H^{\pm} (x, y)   =  \allowbreak O_{\vkappa }  ( \varPi T^{  \vkappa   }  u   )$.  
\end{lem}


\begin{proof}
	
	This lemma follows from shifting the integral contour in \eqref{app: H + (x)} and \eqref{app: H - (x), 2} up to $\mathrm{Re} (it) = 1/2 + \vepsilon$. For more details, see the proof of Lemma 4.1 in \cite{Qi-GL(2)xG(2)-RS}. Note that, for $\mathrm{Re}(\vnu) > -1/2$, we have the Poisson integral representation
	(see  \cite[3.71 (9)]{Watson}): 
	\begin{align*}  
		I_{\vnu} (x) = \frac { 2 ( x/2  )^{\vnu}} {\sqrt{\pi} \Gamma   (\vnu +     1 / 2  )  } \int_0^{\frac 1 2 \pi}   \cosh (x \cos \theta ) \sin^{2 \vnu} \theta \,  {\nd}   \theta . 
	\end{align*}  
\end{proof}

For the rest of this section, let us assume that $ \upgamma (t) $ lies in the space $ \mathscr{T}_{\tau}$. 

\begin{lem}\label{lem: Bessel} 
Let $\upgamma (t) $ be in the space $ \mathscr{T}_{\tau}$.  Let $A \geqslant 0$. Then for $x \Gt T^{-A}$ we have   \begin{equation}\label{4eq: H(x,y) = I(v,w)}
	\begin{split} 
		H^{\pm} (x, y) = \varPi T^{1+\tau}  \int_{-\varPi^{\vepsilon}/\varPi}^{\varPi^{\vepsilon}/ \varPi} 
		g (  \varPi r) \exp (  2i  T r) \cos  ( f^{_{\phantom{i}}}_{^ \pm } (r; v, w) )   \nd r  + O \big(T^{-A}\big)  ,   
	\end{split}
\end{equation}  
where 
\begin{align}\label{4eq: f(r,v,w)}
	f^{_{\phantom{i}}}_{^ \pm } (r; v, w) =  v \exp  (r) \pm   w \exp (- r) , \qquad 	v = \frac {xy} {2}, \quad w = \frac {x/y} 2, 
\end{align}
and  $g (r)$ is a Schwartz function {\rm(}namely, $ g^{(n)} (r) = O ( (|r|+1)^{-m})$ for any $m, n \geqslant 0${\rm)}. 
\end{lem}

	Later in our analysis, we shall always split 
	\begin{align}\label{4eq: cos}
		\cos  ( f^{_{\phantom{i}}}_{^ \pm } (r; v, w) ) = \frac {\exp (i f^{_{\phantom{i}}}_{^ \pm } (r; v, w)) + \exp (- i f^{_{\phantom{i}}}_{^ \pm } (r; v, w) ) } {2},
	\end{align} 
so that we may extract the exponential factor $ \exp (\pm i v) $ or $\exp (\pm i w)$, and we shall usually drop  $g (  \varPi r) \exp (  2i  T r)$ as it does not involve $x, y$ or $v, w$ and is trivially $O (1)$.  

\begin{proof}
This lemma follows essentially from inserting the Mehler--Sonine integrals (see \cite[6.21 (12), 6.22 (13)]{Watson})
	\begin{align*}
		 {J_{2it} (x) - J_{-2it} (x)} 
		& =   \frac {2  \sinh (  \pi t) } {\pi i} \int_{-\infty}^{\infty} \cos (x \cosh r) \exp (2 i t r) \nd r ,  \\
	  K_{2it} (x) & = \frac 1 {2 \cosh (\pi t)}    \int_{-\infty}^{\infty} \cos (x \sinh r) \exp (2itr) \nd r, 
\end{align*} 
into the Bessel integrals as in \eqref{app: H + (x)} and \eqref{app: H - (x)}.  For more details, see the proof of Lemma 4.3 in \cite{Qi-GL(2)xG(2)-RS}. \delete{Note that
\begin{align*}
	x \cosh ( r+ \log y) = v \exp  (r) +   w \exp (- r), \quad x \sinh (r+ \log y) = v \exp  (r) -   w \exp (- r). 
\end{align*}}
\end{proof}


Moreover, Lemma 4.5 in \cite{Qi-GL(2)xG(2)-RS} may be extended easily by applying stationary phase analysis to the integral in \eqref{4eq: H(x,y) = I(v,w)}. 

\begin{lem}\label{lem: analysis of I}
	Let   $T^{\vepsilon} \leqslant \varPi \leqslant T^{1-\vepsilon}$.  
	Then $H^{\pm} (x, y) = O   (T^{-A} )$ if $v, w \Lt T$. 
\end{lem}


By Lemmas \ref{lem: x<1}, \ref{lem: Bessel}, and \ref{lem: analysis of I}, we conclude that if $\upgamma (t)$ is in $ \mathscr{Q}_{\tau} \cap  \mathscr{T}_{\tau} $, then the Bessel integrals $H^{\pm} (x, y)$ are negligibly small for $ xy+x/y \Lt T $. 

\begin{coro}\label{cor: u<T}
Let $\upgamma (t)$ be in the space $ \mathscr{Q}_{\tau} \cap  \mathscr{T}_{\tau} $. Set $u = xy +x/y$ as before. Then for  $ u \Lt T $, we have 
	\begin{align}
		H^{\pm} (x, y) \Lt \varPi T^{\tau} \min \left\{ u,  \frac 1 {T^A}  \right\} ,
	\end{align}
for any $A \geqslant 0$. 
\end{coro}

\begin{remark}\label{rem: gamma (t)}
	Later, in practice, $ \upgamma (t)$ will be those $ \upgamma_{\upsigma}^{\delta} (v, t) $ as defined in \eqref{3eq: defn G1} and \eqref{3eq: defn G2} for $|\mathrm{Im}(v)| \leqslant \log T$ or  simply $\upgamma (t) \equiv 1$ in \S \ref{sec: Bessel, 2}. Note that Definition {\rm\ref{defn: space Q(r)}} requires  $\upgamma   (t) $ to be holomorphic on the strip $ |\mathrm{Im} (t)| \leqslant 1/2 + \vepsilon$.  For the odd case $\delta = 1$, it follows from Lemma \ref{lem: G(v,t)} that $ \upgamma_{\upsigma}^{1} (v, t) $  lies in  $ \mathscr{Q}_{ \upsigma \vepsilon /2} \cap \mathscr{T}_{ \upsigma \vepsilon /2}  $ if we let $\mathrm{Re} (v) = \vepsilon$   in \eqref{3eq: V(y;t), 2}, so we may just choose $\vkappa = \tau = \upsigma \vepsilon /2$.    For the even case $\delta = 0$, however, it is slightly more subtle: In view of Lemma \ref{lem: G(v,t)}, we may apply Corollary \ref{cor: u<T} with $\mathrm{Re} (v) = 1$ and $  \tau = \upsigma /2 $ say, for  $u \Lt T$, and use the integral representation \eqref{4eq: H(x,y) = I(v,w)} in Lemma \ref{lem: Bessel} with $ \mathrm{Re} (v) = \vepsilon$ and $ \tau = \upsigma \vepsilon /2$ for  $u \Gt T$.  
\end{remark}

\begin{remark}
	In {\rm\cite[\S 4]{Qi-GL(2)xG(2)-RS}} we also studied the Bessel transform of a certain $h_{\ssstyle -} (t; y)$ that contains the factor $\upepsilon^{\delta} (t)$   defined similarly to \eqref{6eq: epsilon}, but here we shall not deal with this case {\rm(}see the discussion at the end of \S \ref{sec: AFE}{\rm)}. 
\end{remark}

\begin{remark}
	In the study of related Maass central $L$-values,  the $H^{+}$ and $H^{-}$ Bessel integrals have phases $x \cosh r$ and $x \sinh r$, and hence behave quite differently in \cite{XLi2011} or \cite[\S 7]{Young-Cubic}, but here, as we have just seen,   the phases are $f^{_{\phantom{i}}}_{^ \pm } (r; v, w) = v \exp  (r) \pm   w \exp (- r)$ and may be analyzed more or less in the same way. 
\end{remark}

\section{A Useful Unsmoothing Lemma}

As in \eqref{4eq: defn phi}, for large parameters $T, \varPi$ with $T^{\vepsilon} \leqslant \varPi \leqslant T^{1-\vepsilon}$,  define
\begin{align}\label{8eq: phi}
	\upvarphi (t) =	\upvarphi_{T, \varPi} (t) = \exp \left( - \frac {(t-T)^2}  {\varPi^2} \right). 
\end{align}
In this section, we introduce a useful unsmoothing process as in \cite[\S 3]{Ivic-Jutila-Moments} by an average of the weight $\upvarphi_{T, \varPi} $ in the $T$-parameter.

\begin{defn}
	Let $   H \leqslant T / 3 $ and $T^{\vepsilon} \leqslant \varPi \leqslant H^{1-\vepsilon}  $. Define  
	\begin{align}\label{5eq: defn w(nu)}
		\uppsi (t) = \uppsi_{T, \varPi, H} (t) = \frac 1 { \sqrt{\pi}  \varPi} \int_{\, T- H}^{T + H} \upvarphi_{K, \varPi} (t) \nd K  .
	\end{align}
\end{defn}

The next lemma is essentially   Lemma 5.3 in \cite{Qi-Liu-Moments}, adapted from the arguments in \cite[\S 3]{Ivic-Jutila-Moments}. 

\begin{lem}\label{lem: unsmooth, 1}
	Let  $\lambda $  be a real constant.   	Suppose that $a_f $ is a sequence such that  
	\begin{align}\label{10eq: assumption}
		\sum_{f \in \SB_{\delta}}  \omega_f \upvarphi   (t_f) |a_f|  = O_{\lambda, \vepsilon}  \big( \varPi T^{\lambda}  \big) 
	\end{align}
	for any $\varPi$ with $ T^{\vepsilon} \leqslant \varPi \leqslant T^{1-\vepsilon}  $. Then for  $ \varPi^{1+\vepsilon} \leqslant   H \leqslant T / 3 $ we have
	\begin{align}
		\mathop{\sum  }_{ f \in \SB_{\delta} : \,  |t_f - T| \shskip \leqslant H}  \omega_f  a_f =  \sum_{f \in \SB_{\delta}} \omega_f \uppsi  (t_f) a_f + O_{\lambda, \vepsilon}  \big( \varPi T^{\lambda}  \big) . 
	\end{align}
\end{lem}

Later, we shall choose $a_f$ to be $\delta (L (s_f,   f    ) \neq 0)$ (the Kronecker $\delta$ that detects $L (s_f,   f    ) \neq 0$), $L (s_f,   f    ) $, or $|L (s_f,   f    )|^2$  (see \S \S \ref{sec: conclusion, non-vanishing}, \ref{sec: unsmooth}, and \ref{sec: non-vnaishing, 2}). Note that in all these cases  we  may let  $\lambda = 1$ or $ 1+\vepsilon$ in \eqref{10eq: assumption} in light of the (harmonic weighted) Weyl law or the mean Lindel\"of bounds.

{\large \part{Twisted and Mollified Moments} }

\vspace{5pt}

	\section{Preliminaries}

\subsection{Poisson Summation Formula}

\delete{\begin{lem}
	Let $\phi ( n | c)$ be an arithmetic function of period $c$. Then for $F (x) \in C_c^{\infty} \allowbreak (-\infty, \infty)$ we have 
	\begin{align}\label{4eq: Poisson Cor}
		\sum_{n =-\infty}^{\infty}  \phi (n| c)  F (n) = \frac 1 {c} \sum_{n =-\infty}^{\infty} \hat {\phi }  (n|c) \hat{F} (n/   c), 
	\end{align}
	where $\hat{F} (y)$ is   the   Fourier transform of $ F (x) $ defined by  
	\begin{align}\label{4eq: Fourier}
		\hat {F} (y) = \int_{-\infty}^{\infty} F (x) e (-xy) \nd x , 
	\end{align}
	and
	\begin{align}
		\hat {\phi }  (n|c) =  \sum_{\valpha (\mathrm{mod}\, c) }  \phi (\valpha|c) e \Big(     \frac {   \valpha {n}  } {c} \Big) .  
	\end{align}
\end{lem}}

The following Poisson summation formula is essentially \cite[(4.25)]{IK}. 
\begin{lem}\label{lem: Poisson}  Let $\valpha,  c $ be integers with $c \geqslant 1$. Then for  $F (x) \in C_c^{\infty} (-\infty, \infty)$ we have 
	\begin{align}
		\sum_{n } e \lp  \frac {\valpha n} c \rp F (n) & =   \sum_{n \shskip \equiv \shskip - \valpha (\mod c) }  \hat {F} \Big(  \frac {n} {c} \Big)  ,
	\end{align}
	where $ \hat {F} (y)$ is the Fourier transform of $F (x)$ defined by
	\begin{align*}
		\hat {F} (y) =	\int_{-\infty}^{\infty} F (x) e  ( - x y)   \nd x .
	\end{align*}
\end{lem}

\subsection{Mean-value Bounds for the Riemann $\zeta$ Function} 
\delete{The currently best sub-convexity bound for Riemann's $\zeta (s)$ is due to Bourgain \cite[Theorem 5]{Bourgain}: 
\begin{align}
	\zeta (1/2+it) \Lt  (1+|t| )^{\frac {13}{84} + \vepsilon}.
\end{align}}
The  best bound for the mean square of $ \zeta (1/2+it) $ on short intervals is due to Bourgain and Watt \cite[Theorem 3]{BW-Riemann-2}:
\begin{align}
	\int_{T - U }^{ T + U } |\zeta (1/2+it)|^2 \nd t \Lt U \log T,  \qquad U = T^{\frac {1273} {4053} + \vepsilon} . 
\end{align}
From this and the Cauchy--Schwarz inequality, we deduce easily the next lemma. 
\begin{lem}\label{lem: Riemann}
	Let $1 \leqslant H \leqslant T/3$. We have
	\begin{align}\label{2eq: Riemann, 1}
		\int_{T - H }^{ T + H } |\zeta (1/2+it)|  \nd t \Lt T^{\vepsilon} \big(H^{\frac 1 2 } T^{ \frac {1273} {8106} }   + H \big), 
	\end{align}
	\begin{align}\label{2eq: Riemann, 2}
		\int_{T - H }^{ T + H } |\zeta (1/2+it)|^2 \nd t \Lt T^{\vepsilon} \big(T^{ \frac {1273} {4053} } + H \big). 
	\end{align}
\end{lem}

\subsection{Approximate Functional Equations} 
\label{sec: AFE}

Let $X > 1$ be a large parameter. 
By applying \cite[Theorem 5.3]{IK} to $L ( s  ,   f    )$ at $s = s_f$ and to  $ L (s+it_f ,   f    ) L (s-it_f,   f    ) $ at $s = 1/2$, we deduce from \eqref{3eq: L(s,f)} and \eqref{3eq: Lambda (Q)} the approximate functional equations:  
\begin{align}\label{3eq: AFE} 
	L ( s_f  ,   f    ) =   \!  \sum_{n=1}^{\infty} \!  \frac{ \lambda_f (n)  }{n^{1/2 + i t_f}} V_{1}^{\delta_f}  (   n/X  ; t_f ) + (-1)^{\delta_f} \upepsilon^{\delta_f}  (t_f) \! \sum_{n=1}^{\infty} \!  \frac{ \lambda_f (n)  }{n^{1/2 - i t_f}} V_{1}^{\delta_f}  (   n X  ; - t_f ), 
\end{align}
\begin{align}\label{3eq: AFE, 2} 
	\begin{aligned}
		|L ( s_f  ,   f    )|^2 = 2 \sum_{n_1=1}^{\infty} \sum_{n_2=1}^{\infty}  \frac{ \lambda_f (n_1) \lambda_f (n_2)  }{n_1^{1/2 + i t_f} n_2^{1/2 - i t_f}} V_{2}^{\delta_f} (  n_1 n_2   ; t_f ), 
	\end{aligned}
\end{align}
with
\begin{align}\label{3eq: V(y;t)}
	V_{\upsigma}^{\delta} (y; t) =  \frac 1 {2\pi i}  \int_{(3)} \upgamma_{\upsigma}^{\delta}   (v, t)   y^{-v}  \frac {\nd v} {v} , \qquad \upsigma = 1, 2, 
\end{align}
where
\begin{align}\label{3eq: defn G1}
	\upgamma_{1}^{\delta}  (v, t) =   \exp (v^2) \frac {\Gamma_{\mathbf{R}}^{\delta}  (1/2 + v  )  \Gamma_{\mathbf{R}}^{\delta}  (1/2 + v + 2 it)  } {\Gamma_{\mathbf{R}}^{\delta}  (1/2    )  \Gamma_{\mathbf{R}}^{\delta}  (1/2   + 2 it) } , 
\end{align}
\begin{align}\label{3eq: defn G2}
	\upgamma_{2}^{\delta}  (v, t) =  \exp (v^2) \frac {\Gamma_{\mathbf{R}}^{\delta}  (1/2 + v  )^2 \Gamma_{\mathbf{R}}^{\delta}  (1/2 + v + 2 it) \Gamma_{\mathbf{R}}^{\delta}  (1/2 + v - 2 it)  } { \Gamma_{\mathbf{R}}^{\delta}  (1/2    )^2 \Gamma_{\mathbf{R}}^{\delta}  (1/2   + 2 it) \Gamma_{\mathbf{R}}^{\delta}  (1/2   - 2 it) } ,
\end{align}   
\begin{align}\label{6eq: epsilon}
	\upepsilon^{\delta}  (t) = \frac {\Gamma_{\mathbf{R}}^{\delta}  (1/2   - 2 it)} {\Gamma_{\mathbf{R}}^{\delta}  (1/2  + 2 it)}. 
\end{align}
Moreover, in parallel, we have 
\begin{align}\label{2eq: FE, Eis}
	\zeta (1/2) \zeta (1/2+2it)  = \sum_{n=1}^{\infty}   \frac{ \tau_{  it} (n)  }{n^{1/2 + i t}} V_{1}^{0}  (   n/X  ; t ) +   \upepsilon^{0}  (t)  \sum_{n=1}^{\infty}   \frac{ \tau_{  it} (n)  }{n^{1/2 - i t}} V_{1}^{0}  (   n X  ; - t ),  
\end{align}
\begin{align}\label{2eq: FE, Eis, 2}
	|\zeta (1/2) \zeta (1/2+2it) |^2 = 2 \sum_{n_1=1}^{\infty} \sum_{n_2=1}^{\infty}  \frac{ \tau_{  it} (n_1) \tau_{  it} (n_2)  }{n_1^{1/2 + i t} n_2^{1/2 - i t}} V_{2}^{0} (  n_1 n_2   ; t ), 
\end{align}
up to certain  residual error terms of bounds $ O \big( (|t| + 1)^{1/4+\vepsilon} X^{1/2}  \big) $ and $O \big(\exp (- t^2/2) \big) $ respectively.  

The next two lemmas are standard consequences of the Stirling formula. Compare  \cite[Proposition 5.4]{IK}, \cite[Lemma 1]{Blomer}, and \cite[Lemmas 3.6, 3.7]{Qi-GL(2)xG(2)-RS}.

\begin{lem}\label{lem: AFE}
	Let $U \Gt 1$ and $\theta > 0$.  Then for   $\upsigma = 1, 2$, we have
	\begin{align}\label{3eq: V(y;t), 1}
		{V}_{\upsigma}^{\delta}  (y; t) \Lt_{A  } \bigg(1 + \frac {y} {\sqrt{|t|^{\upsigma} +1 } } \bigg)^{-A},  
	\end{align} 
	for any $A \geqslant 0$,  and 
	\begin{align}\label{3eq: V(y;t), 2}
		{V}_{\upsigma}^{\delta}  (y; t) =  \frac 1 {2\pi i}  \int_{\theta - i U}^{\theta +i U}  \upgamma_{\upsigma}^{\delta}  (v, t)   y^{-v}  \frac {\nd v} {v} + O_{\theta   } \bigg( \frac {(|t|+1)^{\upsigma \theta /2 }  } {y^{\theta} \exp (U^2/2)}  \bigg).
	\end{align} 
\end{lem}

\begin{lem}\label{lem: G(v,t)}
	Let $\mathrm{Re}(v) > - 1/2$ be fixed. Then for  $|\mathrm{Re} (it)| < \mathrm{Re}(v/2) + 1/4 + \delta/2 $, we have
	\begin{align}\label{3eq: bound for G, 1}
		\upgamma_{\upsigma}^{\delta}  (v, t)  \Lt_{ \mathrm{Re} (v)  }    (|t|+1)^{\upsigma    \mathrm{Re} (v/ 2) }  ,  
	\end{align} 
	and, for $t$ real, we have 
	\begin{align}\label{6eq: bound for G, 2}
		\begin{aligned}
			\frac {\partial^{n } \upgamma_{\upsigma }^{\delta}  (v, t) } {\partial t^n  } \Lt_{n,  \mathrm{Re} (v)   }   (|t|+1)^{   {\upsigma} \mathrm{Re} (v/2) }  \lp \frac { \log (|t|+2) } {|t|+1 } \rp^{n}   ,  
		\end{aligned}
	\end{align} 
	for any $n \geqslant 0$. 
\end{lem}

The reader should compare Lemma \ref{lem: G(v,t)} with Definitions \ref{defn: space Q(r)} and \ref{defn: space T(r)}. 

Later, the analysis will be restricted to the range $|t_f-T|, \, |t - T|  \leqslant \varPi \log T$, so it follows from \eqref{3eq: V(y;t), 1} that the second sum  in \eqref{3eq: AFE} or \eqref{2eq: FE, Eis} (after $\upepsilon^{\delta_f} (t_f)$ or $\upepsilon^{0} (t) $) will be negligibly small  if we choose $$X = T^{1/2+\vepsilon},$$
while the first sum  may be truncated effectively at $ n = T^{1+\vepsilon} $. Moreover, the residual error term for \eqref{2eq: FE, Eis} has bound $O (T^{1/2+\vepsilon})$.  Note that the error term in \eqref{3eq: V(y;t), 2} is negligibly small  if we choose $$U = \log T. $$

\subsection{Bi-variable Notation} \label{sec: bi-varialbe}

Subsequently, particularly in our study of the second twisted moment in \S \ref{sec: 2nd moment},  we shall use bold $\boldsymbol{m}$, $\boldsymbol{n}$, ... for the  (positive integral) pairs  $(m_1, m_2)$, $(n_1, n_2)$, .... Define the norm $ \|\boldsymbol{n}\| = n_1 n_2 $, the fraction $\langle \boldsymbol{n} \rangle = n_1/ n_2$,  the dual $\widetilde{\boldsymbol{n}} = (n_2, n_1)$, the reduction $\boldsymbol{n}^{\star} = \boldsymbol{n} / \mathrm{gcd}(\boldsymbol{n}) $,  the product $ \boldsymbol{m n} = (m_1 n_1, m_2 n_2)  $, and the quotient $ \boldsymbol{m}/ \boldsymbol{n} = (m_1 / n_1, m_2 / n_2)  $. Let $\delta (\boldsymbol{n})$ be the Kronecker $\delta$ that detects $n_1 = n_2$.  

Define 
\begin{align}\label{3eq: ct(n)}
	\mathrm{c}_t ( \boldsymbol{n}) = \mathrm{Re} \big((n_1/n_2)^{it} \big) =  \cos (t \log (n_1/n_2)). 
\end{align}
For an arithmetic function $a (n)$, define 
\begin{align}\label{3eq: f(n), bi-var}
	a (\boldsymbol{n}) = a (n_1) a (n_2). 
\end{align} 
For instance,  \eqref{2eq: moments M2} now reads 
\begin{align}\label{3eq: moments M2}
	\begin{aligned}
		\SC^{\delta}_{2}  ( \boldsymbol{m} ) = \sum_{f \in \SB_{\delta} } \omega_f   \lambda_f ( \boldsymbol{m}  ) \mathrm{c}_{t_f} (\boldsymbol{m})   |L (s_f,    f     )|^2     \exp \left(  - (t_f - T)^2/\varPi^2 \right).
	\end{aligned}
\end{align}  
Occasionally, we shall use $\boldsymbol{x} $
for the positive real pair $(x_1 , x_2)$, and, similar to \eqref{3eq: f(n), bi-var}, for a smooth function $\varww (x)$ write 
\begin{align}\label{3eq: w(x)}
	\varww (\boldsymbol{x}) = \varww ( {x}_1 )  \varww ({x}_2) . 
\end{align}


\section{Setup}

For large parameters $T, \varPi$ with $T^{\vepsilon} \leqslant \varPi \leqslant T^{1-\vepsilon}$, let us define 
\begin{align}\label{5eq: phi}
	\upvarphi (t) = \exp \left( - \frac {(t-T)^2}  {\varPi^2} \right),
\end{align}
and consider  
\begin{align}\label{5eq: C1(m)}
	{\SC}^{\delta}_{1\, \natural }  (m) = \sum_{f \in \SB_{\delta} } \omega_f \lambda_f (m)   \big ( m^{-i t_f} L (s_f,    f     )   \upvarphi (t_f) + m^{i t_f} L (\overline{s}_f,    f     )   \upvarphi (- t_f) \big) ,
\end{align}
\begin{align}\label{5eq: C2(m)}
	{\SC}^{\delta}_{2\, \natural }  (\boldsymbol{m} ) = \sum_{f \in \SB_{\delta} }   \omega_f \lambda_f (\boldsymbol{m} ) \mathrm{c}_{ t_f} (\boldsymbol{m})   | L (s_f,    f     )|^2  (  \upvarphi (t_f) +    \upvarphi (- t_f)  ).
\end{align}
Note that 
${\SC}^{\delta}_{1\, \natural }  (m ) $ or $ {\SC}^{\delta}_{2\, \natural } (\boldsymbol{m} )  $ respectively differs from  ${\SC}^{\delta}_{1 } (m  ) $ or $ {\SC}^{\delta}_{2}  (\boldsymbol{m} ) $ as in \eqref{2eq: moments M1} or \eqref{2eq: moments M2} (see also \eqref{3eq: moments M2}) merely by  an exponentially small error term.  In parallel,  the Eisenstein contributions read 
\begin{align}
	\SE_{1}  (m ) = \frac {2}  {\pi}   \cdot  \zeta (1/2)  \int_{-\infty}^{\infty} \omega(t) \tau_{  it} (m) m^{-it}  \zeta (1/2 +2it )  \upvarphi (t)  \nd t ,
\end{align} 
\begin{align}
	\begin{aligned}
		\SE_{2}  (\boldsymbol{m} )	=  \frac {2}  {\pi}   \cdot   \zeta (1/2)^{2}   \int_{-\infty}^{\infty} \omega(t) \tau_{  it} (\boldsymbol{m})   \mathrm{c}_{ t } (\boldsymbol{m})  |\zeta (1/2 +2it )|^{2}   \upvarphi (t)  \nd t ,
	\end{aligned}
\end{align}
where we have used the $t$-evenness  of $\omega (t)$, $ \tau_{it} (m) $, and $ \mathrm{c}_{ t } (\boldsymbol{m}) $. 

\begin{lem}\label{lem: Eis E1,2}
	We have   
	\begin{align}\label{5eq: E1(m)}
		\SE_{1} (m) = O  \big( T^{\vepsilon} \tau (m)   \big(\varPi^{\frac 1 2 } T^{ \frac {1273} {8106} }   + \varPi \big)  \big)  , 
	\end{align}
	\begin{align}\label{5eq: E2(m)}
		\SE_{2} (\boldsymbol{m} )  = O    \big( T^{\vepsilon} \tau  ( \boldsymbol{m} )     \big(T^{ \frac {1273} {4053} } + \varPi \big)  \big) . 
	\end{align}
\end{lem}

\begin{proof}
	Since $ \omega (t) = O ( |t|^{\vepsilon} )$  by \cite[(3.11.10)]{Titchmarsh-Riemann}, it is clear that \eqref{5eq: E1(m)} and \eqref{5eq: E2(m)} follow from \eqref{2eq: Riemann, 1} and \eqref{2eq: Riemann, 2} in Lemma \ref{lem: Riemann} respectively, if we choose $H = \varPi \log T$.   
\end{proof}

\section{The Twisted First Moment}

\subsection{Application of the Kuznetsov Formula} \label{sec: apply Kuz}

By the approximate functional equations \eqref{3eq: AFE} and \eqref{2eq: FE, Eis}, with $$X = T^{1/2+\vepsilon}, $$  we have 
\begin{align*}
	 & {\SC}^{\delta}_{1\, \natural }  (m) = \sum_{n=1}^{\infty}   \frac{1} {  \sqrt{n}}    \sum_{f \in \SB_{\delta} } \omega_f   {  \lambda_f (m    )  \lambda_f (n    )}    {V}_{1}^{\delta } (m,   n ; t_f) + O \big(T^{-A} \big),  \\
	 & {\SE}_{1}  (m)  =    \sum_{n=1}^{\infty}   \frac{1} {  \sqrt{n}}  \cdot   \frac {1} {\pi}   \int_{-\infty}^{\infty} \omega(t) \tau_{  it} (m)  \tau_{  it} (n)   {V}_{1}^{0} (m,  n; t) \nd t + O \big( \varPi T^{1/2+\vepsilon}  \big)  ,  
\end{align*} 
where the errors terms are due to the discussions at the end of \S \ref{sec: AFE}, and  
\begin{align}\label{6eq: V+(n;t)}
	{V}_{1}^{\delta}  (m, n; t) = V_{1}^{\delta}    (  n/ X ; t  )  (m n)^{ - i t}  {\upvarphi (t)}  + V_{1}^{\delta}    (   n/X ; - t  )   (mn)^{  i t} {\upvarphi (- t)} . 
\end{align}  
It follows from the Kuznetsov trace formula in \eqref{2eq: Kuznetsov} that
\begin{align}\label{6eq: C+E=D+O}
	2{\SC}^{\delta}_{1\, \natural }  ( {m}) + 2 (1-\delta) {\SE}_{1}  (m) = \SD_1^{\delta} (m) +   \SO_+^{\delta} (m) + (-1)^{\delta}  \SO_-^{\delta} (m) + O \big(T^{-A}\big) , 
\end{align}
with the diagonal 
\begin{align}\label{6eq: D+-(m)}
	\SD_1^{ {\delta} }   (m) = \frac{1} {\sqrt{m}}    {  {H}_1^{\delta} (m)}  ,   
\end{align} 
where
\begin{align}\label{6eq: H+-(m)}
	 H_1^{\delta}  (m )  =  \frac 1 {\pi^2 }  \int_{-\infty}^{\infty}   {V}_{1}^{\delta}  (m,   m; t) \tanh(\pi t) t \nd t ,
\end{align}
and the off-diagonal 
\begin{align}\label{6eq: O(m)}
	\SO_{\pm}^{ {\delta} } (m) =    \sum_{n=1}^{\infty} \frac{1} {\sqrt{n}} \sum_{c=1}^{\infty} \frac {S( m, \pm n; c)} {c}   {H}_1^{\delta\, \pm} \bigg(\frac {4\pi \sqrt{mn}} {c} ; m,   n \bigg),  
\end{align}
where 
\begin{align}\label{6eq: Bessel, 1+}
	  {H}_1^{\delta\, +} (x; m, n) & = \frac {2i} {\pi}   \int_{-\infty}^{\infty}   {V}_{1}^{\delta}  (m, n; t) J_{2it} (x)  \frac {t \nd t} {\cosh (\pi t) } , \\ 
	  \label{6eq: Bessel, 1-}
	{H}_1^{\delta\, -} (x; m, n) & = \frac {4} {\pi^2}   \int_{-\infty}^{\infty}   {V}_{1}^{\delta}  (m, n; t) K_{2it} (  x)   {\sinh (\pi t)} t \, \nd t . 
\end{align} 

\subsection{Asymptotic for the Diagonal Term}

	In view of \eqref{6eq: V+(n;t)}  and \eqref{6eq: H+-(m)}, we have 
	\begin{align*}
		{H}_{1}^{\delta} (m )  = \frac {2} {\pi^2}  \int_{-\infty}^{\infty}   V_{1}^{\delta}    (   m/ X ; t  ) m^{ - 2 i t}  {\upvarphi (t)}  \tanh(\pi t) t \nd t .
	\end{align*} 
	Keep in mind that  $\upvarphi (t)$ as in \eqref{5eq: phi} is negligibly small unless $|t - T | \leqslant \varPi \log T$. Thus, up to an exponentially small error, $\tanh(\pi t) $ is removable as $\tanh(\pi t) = 1 + O (\exp (-\pi T))$ on this range.   
	
	Let us first evaluate ${\SD}_{1}^{\delta} (1)$.  By shifting the integral contour  in \eqref{3eq: V(y;t)} from $\mathrm{Re} (v) = 3$ down to $\mathrm{Re} (v) = \vepsilon -1/2   $, we have
	\begin{align*}
		V_{1}^{\delta} (1/X; t) = 1 + O \bigg(  \frac 1   { { (X \sqrt{|t|+1}   )^{1/2   - \vepsilon}} } \bigg),  
	\end{align*} 
by the Stirling formula. 	Therefore 
	\begin{align*}
		{H}_{1}^{\delta} (1) = \frac {2} {\pi^2}  \int_{-\infty}^{\infty}     {\upvarphi (t)}  t \nd t + O  (\varPi T^{1/2 +\vepsilon}  ) , 
	\end{align*}
as $X = T^{1/2+\vepsilon}$, whereas 
\begin{align*}
	\int_{-\infty}^{\infty}   {\upvarphi (t)} t \nd t =    \varPi T \int_{-\infty}^{\infty}   \exp (-t^2)  \nd t =  {  \sqrt{\pi}}  \varPi T . 
\end{align*}
Consequently, 
\begin{align*}
	 \SD_{1}^{ \delta } (1 ) = \frac {2} {\pi\sqrt{\pi}} \varPi T   + O   (   {\varPi T^{1/2 + \vepsilon}}      ). 
\end{align*}
Note that in the odd case of $\SD_{1}^{ 1 } (1 )$ the error term may be improved into $O (\varPi / T^{1/2-\vepsilon})$ by shifting the integral contour further down to $\mathrm{Re} (v) = \vepsilon -3/2 $.

Next, we consider the case   $m > 1$. 	By inserting the integral in \eqref{3eq: V(y;t), 2}, with $\theta = \vepsilon$ and  $U = \log T$, we arrive at
\begin{align*}
	 {H}_{1}^{\delta} (m )  = \frac {1} {\pi^3 i}  \int_{\vepsilon-i\log T}^{\vepsilon+i\log T}  \int_{T-\varPi \log T}^{T+\varPi \log T}      \frac { \upgamma_{1}^{\delta} (v, t)    {\upvarphi (t)} t } {  m^{v+2it} / X^{v}  }    \frac {  \nd t \nd v} { v} + O \big(T^{-A} \big) . 
\end{align*} 
For $t $ and $ v$ on the integral domains, in view of  \eqref{3eq: defn G1}, it follows from the Stirling  formula that
\begin{align*} 
	\upgamma_{1}^{\delta} (v, t) = \uprho_{1}^{\delta} (v) \big( t^{v/2} + O  (    {T^{\vepsilon} } / T  ) \big) ,  \qquad \uprho_{1}^{\delta} (v)  =	 \exp \Big(v^2 + \frac {v} 2 \log ( i/\pi ) \Big)   \frac {\Gamma_{\mathbf{R}}^{\delta}  (1/2 + v  )  } {\Gamma_{\mathbf{R}}^{\delta}  (1/2    )} . 
\end{align*}  
Therefore,
\begin{align*} 
	 H_{1}^{\delta} (m)  = \frac {1} {\pi^3 i}  \int_{\vepsilon-i\log T}^{\vepsilon+i\log T}  \frac {\uprho_{1}^{\delta} (v)} {(m/X)^{v} }  \int_{T-\varPi \log T}^{T+\varPi \log T}        {    {\upvarphi (t)} t^{1+v/2} } m^{-2it}   \nd t \, \frac {   \nd v} { v} + O  (  {\varPi T^{\vepsilon}}  ) . 
\end{align*} 
For $m > 1$, the inner $t$-integral is Fourier of phase $ - 2 t \log m$, so by repeated partial integration we infer that the integral is negligibly small. Consequently, we have $ H_1^{\delta} (m) = O (\varPi T^{\vepsilon}) $, and hence
$$ \SD_{1}^{\delta } (m ) =   O \bigg( \frac {\varPi T^{  \vepsilon}} {\sqrt{m} }  \bigg) . $$ 

To summarize, we conclude with the following lemma for $\SD_{1}^{\delta} (m ) $. 

\begin{lem}\label{lem: H+-(m)}   We have
	\begin{align}
		\SD_{1}^{ \delta } (m ) = \frac {2} {\pi\sqrt{\pi}} \varPi T \cdot {\delta ({m, 1})} + O \bigg( T^{\vepsilon}    \varPi  \bigg( \sqrt{T} \delta (m, 1) + \frac 1 {\sqrt{m}} \bigg)    \bigg) . 
	\end{align}
\end{lem}

\delete{\begin{remark}
	Actually, the integral in     may be evaluated explicitly by {\rm\cite[3.896 4, 3.952 1]{G-R}} and there will arise the exponential factor $\exp (-\varPi^2 \log^2 m)$.  
\end{remark}}

\subsection{Estimation of the Off-diagonal Sums}  \label{sec: off-diagonal}

 Let us  consider  the off-diagonal sums $	\SO_{\pm}^{\delta} (m)$ given by  \eqref{6eq: V+(n;t)}, \eqref{6eq: O(m)}, \eqref{6eq: Bessel, 1+}, and \eqref{6eq: Bessel, 1-}.

By \eqref{3eq: V(y;t), 1} in Lemma \ref{lem: AFE} (see also the discussion at the end of \S \ref{sec: AFE}) along with  a smooth dyadic partition,  at the cost of a negligible error,  we may confine the $n$-sum  in \eqref{6eq: O(m)}  on the range $   n \sim N $ for dyadic $ N \Lt T^{1+\vepsilon}$ (namely, $N$ is of the form $2^{j/2}$).  
Next, we insert the expression  of $ V_{1}^{\delta}     (    n / X ; t  ) $ as in \eqref{3eq: V(y;t), 2} in Lemma \ref{lem: AFE} into \eqref{6eq: V+(n;t)},  \eqref{6eq: Bessel, 1+}, and \eqref{6eq: Bessel, 1-}. After this, there arise  the Bessel transforms $H_1^{\delta \, \pm} (x, y; v)$ of the  function  
\begin{align*}
	 h_{1}^{\delta} (t; y; v) = \upgamma_{1}^{\delta} (v, t)  y^{-2it} \upvarphi (t) + \upgamma_{1}^{\delta} (v, - t)  y^{2it} \upvarphi (- t), 
\end{align*} 
as in \S \ref{sec: Bessel}, for 
\begin{align*}
	 x = \frac {4\pi \sqrt{mn}} {c}, \qquad y =   \sqrt{m n} .  
\end{align*}
 Since $\delta$ and $v$ are not essential in our analysis,   they will usually be omitted in the sequel.  
 By Corollary \ref{cor: u<T} and Remark \ref{rem: gamma (t)}, along with the Weil bound \eqref{3eq: Weil}, we may further restrict the $c$-sum  in \eqref{6eq: O(m)}   to the range $c  \Lt  m N/T $. 
It follows that, up to  a negligible error, the off-diagonal sum $ \SO_{\pm}^{\delta} (m)  $ has bound 
\begin{align}\label{7eq: O1(m), 1}
	 \SO_{\pm}^{\delta} (m) \Lt T^{\vepsilon}  \max_{N \Lt T^{1+\vepsilon} }    \frac {1} {\sqrt{N}}   \int_{\vepsilon-i\log T}^{\vepsilon+i\log T}   \big|\SO_{\pm}^{\delta} (m ; N;  v) \big|   \frac {\nd v} {v} , 
\end{align}
where, if we omit  $\delta$ and $v$ from the notation, 
\begin{align}\label{6eq: Opm(m; N)}
\SO_{\pm}^{\delta} (m ; N;  v) =	\SO_{\pm}  (m ; N) = \! \sum_{c \Lt mN/T }     \sum_{n \sim N}    \!   \frac   {S(m, \pm n; c)} c    \varww \Big(  \frac { n} {N} \Big) {H}_{1}^{ \pm }  \bigg(\frac {4\pi \sqrt{mn}} {c} ,   \sqrt{m n}  \bigg), 
\end{align} 
with
\begin{align*}
	\varww (x) = \frac {\varvv (x)} {x^{1/2+v}}, 
\end{align*}
for a certain fixed smooth weight $\varvv  \in C_c^{\infty}[1,2]$. Note that $v$ will not play a role since 
we have uniformly 
\begin{align*}
	\varww^{(j) } (x) \Lt_{j, \vepsilon} \log^{j} T . 
\end{align*}

Next, in \eqref{6eq: Opm(m; N)}, we open  the Kloosterman sum $S (m, \pm n;c)$ (see \eqref{3eq: Kloosterman}) and insert the formula of the Bessel integrals $H_1^{ \pm} (x, y)$ as in  Lemma \ref{lem: Bessel}. For the latter, we have 
\begin{align*}
	v = \frac {2\pi m n} {c}, \qquad w = \frac {2\pi} {c}, 
\end{align*}
then we split  $\cos  ( f^{_{\phantom{i}}}_{^{\pm} } (r; v, w) )$ as in \eqref{4eq: cos} and extract the exponential factor $e (mn/ c)$ or $ e (- mn/ c) $. 
It follows that 
\begin{align}\label{7eq: O1(m), 2}
	\SO_{\pm}  (m ; N) \Lt  {\varPi T^{1+\vepsilon} }  \sum_{\pm}  	    \sum_{c \Lt mN/T }   \frac 1 c   \int_{-\varPi^{\vepsilon}/ \varPi}^{\varPi^{\vepsilon}/ \varPi} \big| \SO{}_{ ^ \pm }^{_ \pm}  (r;   m, c ; N) \big|    \nd r  , 
\end{align} 
where  
\begin{align}
\SO{}_{ ^ \pm  }^{_ \pm} (r;  m, c ; N) = \sumx_{\valpha (\mathrm{mod}\, c) } e   \lp \pm \frac {\widebar{\valpha} m } c  \rp   \sum_{n \sim N }      e \bigg(    \frac {( \valpha + m)  n} {c}   \bigg)   e \Big(       \frac { mn} {c} \uprho (r)  \Big)   \varww^{_{\phantom{i}}}_{^ \pm }  \Big(  \frac { n} {N} \Big), 
\end{align}   
with 
$$\uprho (r) = \exp (r) - 1, $$    $\varww^{_{\phantom{i}}}_{^ +} (x) = \varww (x) $, and $\varww^{_{\phantom{i}}}_{^ -} (x) = \overline{\varww (x) }$. On applying to the $n$-sum the Poisson summation formula in Lemma \ref{lem: Poisson}, we have
\begin{align*} 
	\SO{}_{ ^ \pm  }^{_ \pm} (r;  m, c ; N) & = N \sumx_{\valpha (\mathrm{mod}\, c) } e   \lp \pm \frac {\widebar{\valpha} m } c  \rp   \sum_{n \equiv   - \valpha - m (\mathrm{mod}\, c)}    \hat{\varww}^{_{\phantom{i}}}_{^ \pm }  \bigg(  \frac { N (n-m \uprho (r))} {c}   \bigg) \\
	& = N \sum_{n: \, (n+m, c) = 1}  e \bigg( \! \mp \frac {\overline{  (n+m)} m} {c} \bigg)   \hat{\varww}^{_{\phantom{i}}}_{^ \pm }  \bigg(  \frac { N (n-m \uprho (r))} {c}   \bigg) .
\end{align*} 
Note that the Fourier transform $ \hat{\varww}^{_{\phantom{i}}}_{^ \pm } ( N (n-m \uprho (r))/ c ) $ is negligibly small unless
\begin{align}\label{6eq: |n-mr|}
	| n-m \uprho (r) | < \frac {c T^{\vepsilon}} {N} .  
\end{align}
As 
\begin{align*}
	\frac {c T^{\vepsilon}} {N} \Lt \frac {m T^{\vepsilon}} {T} , \qquad m \uprho (r) \Lt \frac {m \varPi^{\vepsilon}} {\varPi}, 
\end{align*}
we may restrict the $n$-sum to the range $ |n | \Lt    {m \varPi^{\vepsilon}} / {\varPi} $. Consequently, up to a negligible error, 
\begin{align*}
	\int_{-\varPi^{\vepsilon}/ \varPi}^{\varPi^{\vepsilon}/ \varPi} \big| \SO{}_{ ^ \pm }^{_ \pm}  (r;   m, c ; N) \big|    \nd r \Lt N \sum_{|n| \Lt {m \varPi^{\vepsilon}} / {\varPi} }  \int_{-\varPi^{\vepsilon}/ \varPi}^{\varPi^{\vepsilon}/ \varPi} \bigg|\hat{\varww}^{_{\phantom{i}}}_{^ \pm }  \bigg(  \frac { N (n-m \uprho (r))} {c}   \bigg) \bigg| \nd r .
\end{align*}
 Now the $r$-integral may be confined further to the interval defined by \eqref{6eq: |n-mr|} of length $ O (cT^{\vepsilon} / m N) $, and hence
 \begin{align}\label{6eq: bound of integral}
 	 \int_{-\varPi^{\vepsilon}/ \varPi}^{\varPi^{\vepsilon}/ \varPi} \big| \SO{}_{ ^ \pm }^{_ \pm}  (r;   m, c ; N) \big|    \nd r \Lt T^{\vepsilon} N \Big(  1 + \frac {m} {\varPi}    \Big) \frac {c} {m N} = T^{\vepsilon}   \Big(\frac {c} {m} + \frac c {\varPi} \Big). 
 \end{align}
On inserting \eqref{6eq: bound of integral} into \eqref{7eq: O1(m), 2}, we have
\begin{align}\label{6eq: O(m, N)}
	\SO_{\pm}  (m ; N) \Lt T^{\vepsilon} N (  \varPi + m) . 
\end{align}
Note that the implicit constant in \eqref{6eq: O(m, N)} depends on $\vepsilon$ but is uniform for all $ v \in [\vepsilon - i \log T, \vepsilon + i \log T]$. Consequently, from \eqref{7eq: O1(m), 1} and \eqref{6eq: O(m, N)}, we deduce the following bound for $\SO_{\pm}^{\delta} (m)$. 

\begin{lem}\label{lem: O1(m)}
	We have 
	\begin{align}\label{6eq: O1(m), final}
		 \SO_{\pm}^{\delta} (m) \Lt ( \varPi + m) T^{1/2+\vepsilon}. 
	\end{align}
\end{lem}


\subsection{Conclusion}  

In view of \eqref{6eq: C+E=D+O}, Theorem \ref{thm: C1(m)} is a direct consequence of Lemmas \ref{lem: Eis E1,2}, \ref{lem: H+-(m)}  and \ref{lem: O1(m)}.

\section{The Twisted Second Moment} \label{sec: 2nd moment}

As the bi-variable notation will be used frequently in what follows, the reader is advised to read \S \ref{sec: bi-varialbe} before proceeding with this section.

\subsection{Application of Kuznetsov Formula} 

First, in the bi-variable notation, by the approximate functional equations \eqref{3eq: AFE, 2} and \eqref{2eq: FE, Eis, 2} (here, we need to expand $2 \mathrm{c}_t ( \boldsymbol{m}  ) \mathrm{c}_t (  \boldsymbol{n}) = \mathrm{c}_t ( \boldsymbol{m} \boldsymbol{n}) + \mathrm{c}_t ( \boldsymbol{m} / \boldsymbol{n}) $ (see \eqref{3eq: ct(n)}) and use the symmetry in the   $\boldsymbol{n}$-sums),   we obtain 
\begin{align*}  
	 &	{\SC}^{\delta }_{2 \, \natural}  ( \boldsymbol{m})     =   
	  \sum_{\boldsymbol{n}}   \frac{ 1 } {  \sqrt{\| \boldsymbol{n} \|} }    \sum_{f \in \SB_{\delta} }  \omega_f   {  \lambda_f ( \boldsymbol{m}   )  \lambda_f (\boldsymbol{n}  )}     {V}_{2}^{\delta} ( \boldsymbol{m},    \boldsymbol{n}   ; t_f) ,  \\
	 	 {\SE}_{2}  (\boldsymbol{m}) & =     \sum_{\boldsymbol{n}}   \frac{ 1 } {  \sqrt{\| \boldsymbol{n} \|} }   \cdot \frac {1} {\pi} \int_{-\infty}^{\infty}  \omega(t) \tau_{  it} ( \boldsymbol{m} ) \tau_{  it} (\boldsymbol{n})       {V}_{2}^{0}  (\boldsymbol{m},    \boldsymbol{n}  ; t)  \nd t + O \big(T^{-A} \big)  ,  
\end{align*}  
where $ \boldsymbol{n} = (n_1, n_2)$ and $ \| \boldsymbol{n}\| = n_1 n_2 $ as in  \S \ref{sec: bi-varialbe}, and 
\begin{align*}
  {V}_{2}^{\delta} (\boldsymbol{m}, \boldsymbol{n} ; t) =  2  V_2^{\delta} (   \| \boldsymbol{n} \|  ; t) \mathrm{c}_t ( \boldsymbol{m} \boldsymbol{n}) 
 (  \upvarphi (t) +    \upvarphi (- t)  ). 
\end{align*}  
Note that $ {V}_{2}^{\delta} (\boldsymbol{m},  \boldsymbol{n} ; t)$ is even in $t$ and real-valued. 
By the Hecke relation \eqref{3eq: Hecke}, 
\begin{align*}  
&	{\SC}^{\delta }_{2\, \natural}  (\boldsymbol{m})  =  \sum_{\boldsymbol{d} \boldsymbol{h} = \boldsymbol{m} }   
  \sum_{  \boldsymbol{n}}  \frac{ 1 } {  \sqrt{  \| \boldsymbol{d n} \|} }     \sum_{f \in \SB_{\delta} } \omega_f    \lambda_f ( \boldsymbol{h n}  )   {V}_2^{\delta} (\boldsymbol{m}, \boldsymbol{d}   \boldsymbol{  n}  ; t_f) , \\
  {\SE}_{2}  (\boldsymbol{m}) &  =   \sum_{\boldsymbol{d} \boldsymbol{h} = \boldsymbol{m} }   
\sum_{  \boldsymbol{n}}  \frac{ 1 } {  \sqrt{  \| \boldsymbol{d n} \|} } \cdot    \frac {1} {\pi}   \int_{-\infty}^{\infty}   \omega(t)  \tau_{  it}    ( \boldsymbol{h n}  )    {V}_2^0 (\boldsymbol{m}, \boldsymbol{d}   \boldsymbol{  n}  ; t ) \nd t  + O \big(T^{-A} \big)   , 
\end{align*} 
where, according to the bi-variable notation in 
\S \ref{sec: bi-varialbe}, the first summations are over $ \boldsymbol{d} = (d_1, d_2) $ and $ \boldsymbol{h} = (h_1, h_2)$ such that $ (d_1 h_1, d_2 h_2) = (m_1, m_2)$. 
Now an application of the Kuznetsov trace formula in \eqref{2eq: Kuznetsov} yields
\begin{align}\label{7eq: C+E=D+O}
	2 {\SC}^{\delta }_{2\, \natural}  ( \boldsymbol{m}) + 2(1-\delta) {\SE}_{2}  (\boldsymbol{m}) = \SD_2^{\delta } (\boldsymbol{m})  +  \SO_{+}^{\delta } (\boldsymbol{m}) + (-1)^{\delta} \SO_{-}^{\delta } (\boldsymbol{m}) + O \big(T^{-A}\big), 
\end{align}
with the diagonal 
\begin{align}\label{7eq: D2(m)}
	\SD_2^{\delta } (\boldsymbol{m}) = \sum_{\boldsymbol{d} \boldsymbol{h} = \boldsymbol{m} }    \sum_{  \boldsymbol{n}}  \frac{ 1  }  {   \sqrt{  \| \boldsymbol{d n} \|} }  \delta ({\boldsymbol{h n} })  {H}_2^{\delta }  (  \boldsymbol{d }, \boldsymbol{d}   \boldsymbol{ n}), 
\end{align} 
where $ \delta (\boldsymbol{hn}) $ is the Kronecker $\delta$ that detects $ h_1 n_1 = h_2 n_2 $, and, since $ \mathrm{c}_{t} ( \boldsymbol{d^2 h n} ) = \mathrm{c}_{2t} ( \boldsymbol{d} )  $ if   $\delta (\boldsymbol{hn}) = 1$, 
\begin{align}\label{7eq: H2(m)}
	  {H}_2^{\delta } (\boldsymbol{d},   \boldsymbol{  n} ) = \frac 4 {\pi^2 }  \int_{-\infty}^{\infty}  V_2^{\delta } (      \| \boldsymbol{n} \|  ; t)  \mathrm{c}_{2t} ( \boldsymbol{d} ) 
	    \upvarphi (t)   \tanh(\pi t) t \nd t ,
\end{align}
and the off-diagonal 
\begin{align}\label{7eq: O(m)}
	\begin{aligned}
		\SO_{\pm}^{\delta } (\boldsymbol{m}) = \sum_{\boldsymbol{d} \boldsymbol{h} = \boldsymbol{m} }    \sum_{  \boldsymbol{n}}  \frac{ 1  } {  \sqrt{  \| \boldsymbol{d n} \|} }     \sum_{c  } \frac{S( \boldsymbol{1}_{\pm}  \boldsymbol{h n}   ; c)} {c}   {H}_2^{\delta\, \pm} \bigg(\frac{4\pi\sqrt{\| \boldsymbol{h n} \|  }}{c};  \boldsymbol{m},  \boldsymbol{d}    \boldsymbol{n} \bigg) ,
	\end{aligned}
\end{align}
where  $\boldsymbol{1}_{\pm} = (1, \pm 1)$, hence $S ( \boldsymbol{1}_{\pm}  \boldsymbol{h n}   ; c) = S (h_1 n_1, \pm h_2 n_2; c)$, and  
\begin{align}\label{7eq: H+(x)}
	  {H}_2^{\delta\, +} ( x; \boldsymbol{m},   \boldsymbol{n} ) & = \frac {4 i} {\pi}   \int_{-\infty}^{\infty}    	 {V}_2^{\delta} (     \| \boldsymbol{n} \| ;  t)    \mathrm{c}_t (   \boldsymbol{m n})
	 (  \upvarphi (t) +    \upvarphi (- t)  ) J_{2it} (x)  \frac {t \nd t} {\cosh (\pi t) }, \\
\label{7eq: H-(x)}
{H}_2^{\delta\, -} ( x; \boldsymbol{m},   \boldsymbol{n} ) & = \frac {8} {\pi^2}   \int_{-\infty}^{\infty}    	 {V}_2^{\delta} (     \| \boldsymbol{n} \| ;  t)    \mathrm{c}_t (   \boldsymbol{m n})
(  \upvarphi (t) +    \upvarphi (- t)  ) K_{2it} (  x)   {\sinh (\pi t)} t   \nd t	 .
\end{align}

\subsection{Asymptotic for the Diagonal Sum}

Set  
\begin{align*}
  h = \mathrm{gcd}(\boldsymbol{h}), \qquad    \boldsymbol{h^{\star}} = \boldsymbol{h} / h; 
\end{align*}
then $  \delta ({\boldsymbol{h n} }) = 1 $  if and only if there is $n$ such that 
 $	\boldsymbol{n} = n   \widetilde{\boldsymbol{h^{\star} }}  .  $ 
Recall from \S \ref{sec: bi-varialbe} that $ \widetilde{\boldsymbol{h^{\star}}} $ is  dual to $\boldsymbol{h^{\star}}$. Therefore \eqref{7eq: D2(m)} is turned into 
\begin{align}\label{7eq: D2(m), 2.0}
	\SD_2^{\delta} (\boldsymbol{m}) = \sum_{\boldsymbol{d} \boldsymbol{h} = \boldsymbol{m} }    \sum_{   {n}}  \frac{ 1 }  { n \sqrt{\| \boldsymbol{d h^{\star} }  \|}   }     {H}_2^{\delta}  (  \boldsymbol{d }, n   \boldsymbol{d} \widetilde{ \boldsymbol{h^{\star}} }   ) . 
\end{align} 
In view of \eqref{3eq: V(y;t), 1} and \eqref{7eq: H2(m)}, the summation  in \eqref{7eq: D2(m), 2.0} may be restricted effectively to $ n    \sqrt{\| \boldsymbol{m} \|} / h  \leqslant T^{1/2+\vepsilon}$. It follows from \eqref{3eq: V(y;t), 1} and  \eqref{5eq: phi} that $ H_2^{\delta} (\boldsymbol{d}, \boldsymbol{n}) = O (\varPi T) $, so trivial estimation   yields 
\begin{align*}
\SD_2^{\delta}	(\boldsymbol{m}) = O \bigg(  \frac { \varPi T^{1+\vepsilon} } {\sqrt{\|\boldsymbol{m^{\star}}\|} } \bigg), 
\end{align*}
where $\boldsymbol{m^{\star}} = \boldsymbol{m} / \mathrm{gcd} (\boldsymbol{m})$. 
However, our aim here is to derive an asymptotic formula for $\SD_2^{\delta}	(\boldsymbol{m})$. 

\begin{lem}\label{lem: H2(m,n)}
	We have 
	\begin{align}
	 {H}_2^{\delta} (\boldsymbol{d},   \boldsymbol{  n} )  = O    \bigg(  \frac {\varPi T^{\vepsilon}}  { \| \boldsymbol{n} \|^{\vepsilon}}  \bigg), 
	\end{align}
if 
\begin{align}\label{7eq: h1/h2 - 1}
	 | \langle \boldsymbol{d} \rangle - 1  | >  \frac {\varPi^{\vepsilon}} {\varPi}, \qquad \text{{\rm(}$ \langle \boldsymbol{d} \rangle = d_1/d_2${\rm)}} . 
\end{align}
\end{lem}

\begin{proof}
By \eqref{3eq: V(y;t), 2}, with $\theta = \vepsilon$ and $U = \log T$, we have 
	 \begin{align*}
	 {H}_2^{\delta} (\boldsymbol{d},   \boldsymbol{  n} )   = \frac {2} {\pi^3 i}  \int_{\vepsilon-i\log T}^{\vepsilon+i\log T}  \int_{T-\varPi \log T}^{T+\varPi \log T}      \frac { \upgamma_{2}^{\delta} (v, t)  \mathrm{c}_{2t} ( \boldsymbol{d}  )   {\upvarphi (t)} t } {  \| \boldsymbol{n} \|^v }    \frac {  \nd t \nd v} { v} + O \big(T^{-A} \big) . 
	 \end{align*}
 For $t$ and $v$ on the integral domains, in view of \eqref{3eq: defn G2}, we have by the Stirling formula that 
 \begin{align*} 
 	\upgamma_{2}^{\delta} (v, t) = \uprho_{2}^{\delta} (v) \big( t^{  v} + O  (    {T^{\vepsilon} } / T  ) \big) ,  \qquad 
 	\uprho_{2}^{\delta} (v) =	 \exp (v^2   ) \pi^{-v}   \frac {\Gamma_{\mathbf{R}}^{\delta}  (1/2 + v  )^2  } {\Gamma_{\mathbf{R}}^{\delta}  (1/2    )^2 }  . 
 \end{align*}
Consequently,
\begin{align*}
	{H}_2^{\delta} (\boldsymbol{d},   \boldsymbol{  n} )   = \frac {2} {\pi^3 i}  \int_{\vepsilon-i\log T}^{\vepsilon+i\log T} \frac { \uprho_{2}^{\delta} (v)  } {  \| \boldsymbol{n} \|^v } \int_{T-\varPi \log T}^{T+\varPi \log T}        {   \mathrm{c}_{2t} ( \boldsymbol{d}  )   {\upvarphi (t)} t^{1+ v} }   \nd t  \frac {   \nd v} { v} + O \bigg(\frac {\varPi T^{\vepsilon}} { \| \boldsymbol{n} \|^{\vepsilon}} \bigg) . 
\end{align*}
The inner $t$-integral is a cosine Fourier integral of phase $ 2 t \log \, \langle \boldsymbol{d} \rangle $ (see \eqref{3eq: ct(n)}), so, by repeated partial integration, we infer that it is negligibly small since $ | \log \, \langle \boldsymbol{d} \rangle | \Gt \varPi^{\vepsilon} / \varPi $ by \eqref{7eq: h1/h2 - 1}.   
\end{proof}

By Lemma \ref{lem: H2(m,n)}, the off-diagonal terms with $ \delta (\boldsymbol{d}) = 0 $ contribute  $ O (\varPi T^{\vepsilon}  / \sqrt{\|\boldsymbol{m^{\star}}\| }) $, since  our assumption $ \max \{ \boldsymbol{m} \} < \varPi^{1-\vepsilon}$ yields  $ |\langle \boldsymbol{d} \rangle - 1| > \varPi^{\vepsilon}/ \varPi $.    

It is left to consider those diagonal terms with $ \delta (\boldsymbol{d}) = 1 $.  Note that in this case $ \boldsymbol{h^{\star}} = \boldsymbol{m^{\star}} $. 
More explicitly, \eqref{7eq: D2(m), 2.0} is turned into
\begin{align}\label{7eq: D2(m), 3.0}
	 \SD_{2}^{\delta} (\boldsymbol{m}) = \frac 1 {\sqrt{\| \boldsymbol{m^{\star} }  \|}} \sum_{d | \mathrm{gcd}(\boldsymbol{m})} \sum_{n}   \frac{ 1 }  { d n    }     {H}_2^{\delta}  (  d n  \sqrt{\| \boldsymbol{m^{\star} }  \|}   )  +   O \bigg(    \frac {T^{1+\vepsilon} } {\sqrt{\|\boldsymbol{m^{\star}}\|} } \bigg) ,  
\end{align} 
where, after the removal of $\tanh (\pi t)$, 
\begin{align}\label{7eq: H2(n)}
	H_2^{\delta} (  x ) = \frac 4 {\pi^2 }  \int_{-\infty}^{\infty}  V_2^{\delta} (   x^2 ; t) 
	\upvarphi (t)   t \nd t . 
\end{align} 
As in \eqref{1eq: m, r}, we set
\begin{align*}
	m = \mathrm{gcd}(\boldsymbol{m}), \qquad r =  \| \boldsymbol{m^{\star} }  \|. 
\end{align*}
It follows from \eqref{3eq: V(y;t)},   \eqref{7eq: D2(m), 3.0}, and \eqref{7eq: H2(n)} that 
\begin{align}\label{7eq: D2(m), 1}
	\SD_{2}^{\delta} (\boldsymbol{m}) = \frac {2} {\pi^3 i} \sum_{d | m} \int \upvarphi (t) t \int_{(3)} \upgamma_2^{\delta} (v, t) \frac {\zeta (1+ 2v)} {(d \sqrt{r})^{1+2v} } \frac {\nd v \nd t} {v}    +   O \bigg(    \frac {T^{1+\vepsilon} } {\sqrt{r} } \bigg) .   
\end{align}  
Now we shift the integral contour from $\mathrm{Re}(v) = 3$ down to $\mathrm{Re} (v) = \vepsilon -1/2$ and collect the residue at the double pole $v = 0$. In view of  \eqref{3eq: defn G2} and \eqref{3eq: bound for G, 1}, the resulting integral has bound  $O ((|t|+1)^{\vepsilon-1/2}  )$ and it yields the error $O (\varPi T^{1/2+\vepsilon})$, while we have Laurent expansions (at $v = 0$): 
\begin{align*}
	\upgamma_2^{\delta} (v, t) =  1 + (\upsigma_2^{\delta} (t) - 4\log \pi)   \frac {v} 2 + \cdots, \qquad   
	\zeta (1+2v) = \frac 1 {2 v} + \gamma + \cdots ,
\end{align*}
\begin{align*}
	 \upsigma_2^{\delta} (t) = 2 \psi ( \vkappa )   +   \psi   ( \vkappa + it  )  + \psi   ( \vkappa - it  ) , \qquad  \vkappa = \frac {1+2 \delta} 4 ,  
\end{align*}
where $\gamma$ is the Euler constant and $\psi (s) = \Gamma'(s)/\Gamma (s)$ is the di-gamma function, and it follows that the residue is equal to 
\begin{align*}
\frac 1 {d \sqrt{r}}\bigg(	\gamma + \frac 1 4 \upsigma_2^{\delta} (t) -   \log (\pi  d \sqrt{r}) \bigg). 
\end{align*}
By the Stirling formula,
\begin{align*}
	\upsigma_2^{\delta} (t) = 2 \psi  (  \vkappa ) + 2 \log \sqrt{ \vkappa^2  + t^2 } + O \bigg(\frac 1 {1+t^2} \bigg). 
\end{align*}
Next, the $t$-integrals may be evaluated easily by 
\begin{align*}
	\int_{-\infty}^{\infty}   {\upvarphi (t)} t \nd t = {  \sqrt{\pi}}  \varPi T , \qquad \int_{-\infty}^{\infty}   {\upvarphi (t)} t \log \sqrt{\vkappa^2+ t^2}  \nd t = {  \sqrt{\pi}}  \varPi T \log T   + O \bigg(\frac {\varPi^3} T  \bigg) .
\end{align*}
It is clear that  the error after this process is $O (\varPi^3T^{\vepsilon}/ T {\displaystyle \sqrt{r     }})$, while the main term reads 
\begin{align*}
	\frac {4} {\pi\sqrt{\pi r}} \varPi T \sum_{d |m} \frac 1 {d}  \bigg(  \gamma +  \frac 1 2 \psi (\vkappa)  +   \log \frac {\sqrt{T}} {\pi d  \sqrt{r}}\bigg)  . 
\end{align*} 
Moreover, by the Gauss  di-gamma formula in \cite[\S 1.2]{MO-Formulas},
\begin{align*}
	\psi \bigg(\frac 1 4 \bigg) = - \gamma - 3\log 2 - \frac {\pi} 2, \qquad  \psi \bigg(\frac 3 4 \bigg) = - \gamma - 3\log 2 + \frac {\pi} 2. 
\end{align*}
Finally, we conclude with the following asymptotic formula for $\SD_{2}^{\delta} (\boldsymbol{m})  $.

\begin{lem}\label{lem: main term}
 Set \begin{align}
	 m = (m_1, m_2), \qquad r      = \frac {m_1m_2}  {m^2}. 
\end{align} 
Define  
\begin{align}\label{7eq: gamma0}
	\gamma_{0 }  =    \gamma   - \log 8 \pi^2 - \frac {\pi} 2 ,  \qquad \gamma_1 = \gamma   - \log 8 \pi^2 + \frac {\pi} 2, 
\end{align}  
\begin{align}\label{7eq: Sigma}
	\sigma_{-1}   ( m    ) =	 {\sum_{d  | m     }  }  
	\frac {1} {d }, \qquad  \sigma_{-1}'   ( m  ) =	 {\sum_{d  | m      }  }  
	\frac {\log d } {d } . 
\end{align}  
Then 
	 \begin{align}\label{7eq: D2, final}
	 \begin{aligned}
	 		 \SD_{2}^{\delta} (\boldsymbol{m}) = \frac { 2  } {\pi \sqrt{\pi r} } \varPi T \bigg(  \bigg( & \log  {\frac T r}   + \gamma_{\delta}   \bigg) \sigma_{-1} (m)  - 2 \sigma_{-1}' (m)  \bigg)   \\
	 	   & + O \bigg(  T^{\vepsilon} \bigg(  \varPi \sqrt{T} +  \frac { T  } {\sqrt{ r      } } +   \frac {  \varPi^3  } {T \sqrt{ r      } }  \bigg) \bigg)   .
	 \end{aligned}
	 \end{align}

\end{lem}

\subsection{Estimation of the Off-diagonal Sums}  \label{sec: off}


Similar to the initial arguments for $ \SO_{\pm}^{\delta} (m)$  in \S \ref{sec: off-diagonal}, by applying Lemma   \ref{lem: AFE}, Corollary \ref{cor: u<T}, and Remark \ref{rem: gamma (t)} to \eqref{7eq: O(m)}, \eqref{7eq: H+(x)}, and \eqref{7eq: H-(x)}, we have 
\begin{align}\label{7eq: O2(m)}
	{\SO}_{\pm}^{\delta } (\boldsymbol{m} ) \Lt T^{\vepsilon}  \max_{\| \boldsymbol{N} \| \Lt T^{1+\vepsilon} }  \frac {1} {\sqrt{\| \boldsymbol{N} \|}}   \sum_{\boldsymbol{d} \boldsymbol{h} = \boldsymbol{m} }   \int_{\vepsilon-i\log T}^{\vepsilon+i\log T}   \big|{\SO}_{\pm}^{\delta } (\boldsymbol{d}, \boldsymbol{h} ; \boldsymbol{N};  v) \big|   \frac {\nd v} {v} , 
\end{align}
where  $\boldsymbol{N} $ are dyadic, 
and   
\begin{align}
\begin{aligned}
	{\SO}_{\pm}  (\boldsymbol{d}, \boldsymbol{h} ; \boldsymbol{N} ) =    \sum_{  c \shskip \Lt C  }     \sum_{ \boldsymbol{n} \sim \boldsymbol{N} / \boldsymbol{d} }  
  \frac  {S( \boldsymbol{1}_{\pm}  \boldsymbol{h n}   ; c)}  {c}  \varww \bigg(    \frac {\boldsymbol{d}  \boldsymbol{n}} {\boldsymbol{N}}   \bigg)    {H}_2^{   \pm}  \bigg(    \frac{4\pi\sqrt{\| \boldsymbol{h n} \|  }}{c}, \langle \boldsymbol{d  } \rangle  \sqrt{\langle \boldsymbol{h n} \rangle }    \bigg)   , 
\end{aligned}
\end{align} 
in which 
\begin{align}
	C = \frac  { \max  \{    \boldsymbol{  m N }   \} } {T \|\boldsymbol{d}\| } , 
\end{align}  
the weight $\varww (\boldsymbol{x})$ is of the product form as in \eqref{3eq: w(x)} for the weight function $\varww   \in C_c^{\infty} [1, 2]$ as in \S \ref{sec: off-diagonal}, and  $H_{2}^{   \pm} (x, y ) $ are the Bessel transforms of 
\begin{align*}
	h_{2}  (t; y ) = 2 \upgamma_{2}^{\delta} (v, t)  \cos (2t \log y) (\upvarphi (t) +  \upvarphi (- t)),
\end{align*}   
as in \eqref{4eq: h2(t;y)}; here $\delta$ and $v$ have been suppressed again from the notation.

By the formula of the Bessel integrals $H^{   \pm} (x, y ) $   as in  Lemma \ref{lem: Bessel}, along with  \eqref{4eq: H2=H} and  \eqref{4eq: cos},  it follows that ${\SO}_{\pm}  (\boldsymbol{d}, \boldsymbol{h} ; \boldsymbol{N} ) $ splits into four similar parts, one of which is  
\begin{align}\label{8eq: O2++}
 \SO^{_{++}}_{^ \pm} (\boldsymbol{d}, \boldsymbol{h} ; \boldsymbol{N} ) \Lt 
\varPi T^{1+\vepsilon}  \sum_{  c \shskip \Lt C  } \frac 1 c \int_{-\varPi^{\vepsilon}/ \varPi}^{\varPi^{\vepsilon}/ \varPi}   \big|
 \SO^{_{++}}_{^ \pm} (r; c; \boldsymbol{d}, \boldsymbol{h} ; \boldsymbol{N} )  \big|    \nd r  , 
\end{align}
with  
\begin{align}\label{8eq: O++}
\SO^{_{++}}_{^ \pm} (r; c; \boldsymbol{d}, \boldsymbol{h} ; \boldsymbol{N} ) =     \sum_{ \boldsymbol{n}  \sim \boldsymbol{N} / \boldsymbol{d} }     {T_{\pm} ( \boldsymbol{h n} ; \boldsymbol{d}   ; c)} e ( \uprho^{_{\phantom{i}}}_{^ \pm} (r; \boldsymbol{h n} , \boldsymbol{d}; c ) )  \varww \bigg(    \frac {\boldsymbol{d} \boldsymbol{n}} {\boldsymbol{N}}   \bigg) , 
\end{align}
where
\begin{align}\label{8eq: defn T}
	T_{\pm} ( \boldsymbol{ n} ; \boldsymbol{d}   ; c) = S(n_1, \pm n_2; c ) e \bigg(\frac {d_1 n_1} {c d_2  } \pm \frac {d_2 n_2} {c d_1  } \bigg), 
\end{align}
\begin{align}\label{8eq: rho}
	\uprho^{_{\phantom{i}}}_{^ \pm} (r; \boldsymbol{ n} , \boldsymbol{d} ; c) = \frac {d_1 n_1} {c d_2} \uprho_1 (r) \pm \frac {d_2 n_2} {c d_1} \uprho_2 (r),  
\end{align}
for 
\begin{align*}
	\uprho_{1} (r) = \exp (  r) - 1, \quad  \uprho_{ 2} (r) = \exp (- r) - 1. 
\end{align*}
Note that the other $\SO^{_{\pm \pm}}_{^ \pm} (r; c; \boldsymbol{d}, \boldsymbol{h} ; \boldsymbol{N} ) $ are defined similarly with $ \varww $ changed into $\varww^{_{\phantom{i}}}_{^\pm}   $ ($\varww^{_{\phantom{i}}}_{^+}   = \varww  $ and $\varww^{_{\phantom{i}}}_{^-}  = \overline{\varww  }$ as in  \S \ref{sec: off-diagonal}) or with $ \uprho_{1}   $ and $ \uprho_{2}   $ interchanged.

Now let us assume $ m_1 N_1 \geqslant m_2 N_2 $, so that 
\begin{align}\label{7eq: C}
	C = \frac {m_1 N_1} {d_1 d_2 T}. 
\end{align}
Our strategy is to apply Poisson summation to the $n_1$-sum but trivial estimation to the $n_2$-sum. Note that in the worst case scenario, the $n_2$-sum could have bounded length, rendering trivial estimation optimal. 

More explicitly, after opening the Kloosterman sum $S (n_1, \pm n_2 ; c)$   in \eqref{8eq: defn T} into an $\valpha$-sum modulo $c$  (see \eqref{3eq: Kloosterman}), the $\valpha$- and $n_1$-sums read
\begin{align*} 
\sumx_{   \valpha      (\mathrm{mod} \, c) } e \lp \pm \frac {\widebar{\valpha} h_2 n_2} {c}  \rp		\sum_{n_1 \sim N_1 / d_1 }    e \bigg( \frac {\valpha h_1 n_1} {c} + \frac {m_1   n_1} {c d_2  } \bigg) e \bigg(   \frac {m_1   n_1} {c d_2  } \uprho_1 (r)  \bigg) \varww \bigg(\frac {d_1  n_1} {N_1} \bigg) .  
	\end{align*}   
By applying the Poisson summation formula modulo $c d_2$ as in Lemma \ref{lem: Poisson}, this is turned into 
\begin{align}\label{9eq: after Poisson}
\frac {N_1} {d_1} \,	\sumx_{   \valpha      (\mathrm{mod} \, c) } e \lp \pm \frac {\widebar{\valpha} h_2 n_2} {c}  \rp		\sum_{n_1 \equiv - \valpha d_2 h_1 - m_1 (\mathrm{mod}\, c d_2)}  \hat{\varww} \bigg( \frac {N_1 (n_1 - m_1 \uprho_{1} (r))} {c d_1 d_2} \bigg). 
\end{align}
Note that the Fourier transform  is negligibly small unless
\begin{align}\label{7eq: |n-mr|}
	| n_1-m_1 \uprho_1 (r) | < \frac {c d_1 d_2 T^{\vepsilon}} {N_1} .  
\end{align}
In view of \eqref{8eq: O2++} and \eqref{7eq: C}, 
\begin{align*}
	\frac {c d_1 d_2 T^{\vepsilon}} {N_1} \Lt \frac {m_1 T^{\vepsilon}} {T}, \qquad m_1 \uprho_{1} (r) \Lt \frac {m_1 \varPi^{\vepsilon}} {\varPi}, 
\end{align*}
so the $n_1$-sum may be restricted effectively to the range $|n_1| \Lt m_1 \varPi^{\vepsilon}/\varPi $. However, our assumption $ m_1 < \varPi/\varPi^{\vepsilon} $ implies that only the zero frequency survives if we adjust $\vepsilon$ by our $\vepsilon$-convention. 

Given that $n_1 = 0$,  from \eqref{9eq: after Poisson} and \eqref{7eq: |n-mr|} we obtain the congruence equation
\begin{align*}
	 \valpha d_2 h_1  \equiv  - m_1 (\mathrm{mod}\, c d_2 ), 
\end{align*}
and the shrunk integral domain  
\begin{align}\label{7eq: bound for |r|}
	| r | < \frac {c d_1 d_2 T^{\vepsilon}} {m_1 N_1 } . 
\end{align}
Note that the former has at most $ (c,   h_1) $ many solutions of $\valpha$ modulo $c$. Moreover, a trivial estimation of the $n_2$-sum yields a factor $O (N_2/d_2)$. Consequently,  
\begin{align}\label{7eq: O++}
	\big|
	\SO^{_{++}}_{^ \pm} (r; c; \boldsymbol{d}, \boldsymbol{h} ; \boldsymbol{N} )  \big|    \Lt \frac {N_1 N_2 (c,   h_1)} {d_1 d_2}     . 
\end{align}
Next, we insert this into \eqref{8eq: O2++} and bound the $r$-integral by $O (cd_1 d_2 T^{\vepsilon} / m_1 N_1)$, which is the length of the new integral domain defined by \eqref{7eq: bound for |r|}.  At any rate, we have
\begin{align*}
	\SO^{_{++}}_{^ \pm} (\boldsymbol{d}, \boldsymbol{h} ; \boldsymbol{N} ) & \Lt T^{\vepsilon} \cdot \varPi T \cdot \frac {N_1 N_2} {d_1 d_2} \cdot  \frac {d_1 d_2} {m_1 N_1}  \sum_{c \Lt {m_1 N_1} / {d_1 d_2 T}}    {(c,   h_1)}     \\
	& \Lt   T^{\vepsilon} \frac { \varPi   N_1 N_2} {d_1 d_2}. 
\end{align*}
Of course, the bound in the case $ m_1 N_1 < m_2 N_2 $ is the same, and hence 
\begin{align}\label{7eq: O(d,h;N)}
	\SO_{\pm}  (\boldsymbol{d}, \boldsymbol{h} ; \boldsymbol{N} ) \Lt   T^{\vepsilon} \frac {  \varPi  \|\boldsymbol{N}\|} {\|\boldsymbol{d}\|}. 
\end{align}
Thus by \eqref{7eq: O2(m)} and \eqref{7eq: O(d,h;N)}, we conclude with the following estimate for $ {\SO}_{\pm}^{\delta } ( \boldsymbol{m} )  $.

\begin{lem}\label{lem: O2(m)}
	We have 
	\begin{align}
		{\SO}_{\pm}^{\delta } ( \boldsymbol{m} ) \Lt     { \varPi   T^{1/2+\vepsilon}} . 
	\end{align}
\end{lem}

\subsection{Conclusion} In view of \eqref{7eq: C+E=D+O}, Theorem \ref{thm: C2(m)}   is a direct consequence of Lemmas 
\ref{lem: Eis E1,2}, \ref{lem: main term}, and \ref{lem: O2(m)}.

\section{The Mollified Moments} \label{sec: mollified}

Historically, the mollifier method was first used  on zeros of the Riemann zeta  function by Bohr, Landau, and Selberg \cite{Bohr-Landau,Selberg-Mollifier}. In this section, we shall mainly follow the ideas and maneuvers in \cite{IS-Siegel,KM-Analytic-Rank-2}.

\subsection{Setup} 

We introduce the mollifier  
\begin{align}\label{9eq: defn Mj}
	M_f	= \sum_{m \leqslant M}  \frac {x_m     \lambda_f  (m)} {m^{1/2+it_f}}  . 
\end{align} 
Assume that the  coefficients $ x_m  $ 
are bounded, real-valued, normalized so that 
\begin{align}
	\label{9eq: x1=1}
	x_1 = 1, 
\end{align}
supported on square-free $m $ with  $m \leqslant M$. 
Note that
\begin{align}\label{9eq: |Mj|2}
	|M_f|^2 = \mathrm{Re} \mathop{\sum \sum}_{m_1, m_2 \leqslant M}  \frac {x_{m_1} x_{m_2}     \lambda_f (m_1) \lambda_f {(m_2)}  } {\sqrt{m_1 m_2} (m_1/m_2)^{it_f} } . 
\end{align}
For $\upvarphi (t)$ given in \eqref{5eq: phi}, consider the mollified moments
\begin{align}\label{9eq: defn M1}
	\SM_1^{\delta} =  \sum_{f \in \SB_{\delta}}   \omega_f  M_f   L (s_f,    f     )  \upvarphi (t_f), 
\end{align} 
\begin{align}\label{9eq: defn M2}
	\SM_2^{\delta} =  \sum_{f \in \SB_{\delta} }   \omega_f  |M_f|^2     |L (s_f,    f     )|^2    \upvarphi (t_f). 
\end{align} 
By the Cauchy inequality, 
\begin{align}\label{9eq: Cauchy} 
		\sumd_{ L (s_f,   f    ) \neq 0} \omega_f \upvarphi (t_f)   \geqslant \frac { \left|    {\displaystyle \sumd }    \omega_f  M_f L (s_f,   f    )  \upvarphi (t_f)   \right|^2  }  { {\displaystyle \sumd }    \omega_f \left|M_f  L (s_f,   f    ) \right|^2 \upvarphi (t_f)  } ,
\end{align}
where, of course,   the $\delta$   indicates summation over $\SB_{\delta}$. 
Our task is to calculate the mollified moments $  \SM_1^{\delta}$ and $\SM_2^{\delta} $ and then carefully choose the coefficients $x_m$ to maximize the ratio $\big|\SM_1^{\delta} \big|^ 2  / \SM_2^{\delta} $.

\subsection{The Mollified First Moment}

By the definitions in \eqref{2eq: moments M1}, \eqref{9eq: defn Mj}, and \eqref{9eq: defn M1}, we have
\begin{align*}
	\SM_1^{\delta} = \sum_{m \leqslant M} \frac {x_m  } {\sqrt{m}} \SC_1^{\delta} (m). 
\end{align*}
Thus if we insert the  asymptotic formula of $\SC_1^{\delta} (m)$ in Theorem \ref{thm: C1(m)}, then  we obtain the main term $\varPi T /\pi \sqrt{\pi}$, while, as $x_m$ are bounded, we estimate the error term by 
\begin{align*}
	\varPi T^{1/2+\vepsilon} \sum_{m \leqslant M}  \frac {1} {\sqrt{m}} + T^{1/2+\vepsilon} \sum_{m \leqslant M}     {\sqrt{m}} \Lt T^{\vepsilon} \sqrt{MT} \big(\varPi  + M \big) . 
\end{align*}

\begin{lem}\label{lem: M1}
Let $\varPi = T^{ \vnu }$    and  $M = T^{\varDelta}$. Then for $\vnu < 1$ and   $\varDelta <   (2\vnu + 1)/3$ there exists $\vepsilon = \vepsilon (\vnu, \varDelta) > 0$ such that 
\begin{align}\label{9eq: M1}
	\SM_1^{\delta} = \frac { \varPi T} {\pi \sqrt{\pi}} + O \big(\varPi T^{1-\vepsilon} \big). 
\end{align}
\end{lem}

\subsection{The Mollified Second Moment}

In view of the definitions in \eqref{2eq: moments M2}, \eqref{9eq: |Mj|2}, and \eqref{9eq: defn M2}, we have
\begin{align*}
	\SM_2^{\delta} =  \mathop{\sum \sum}_{m_1, m_2 \leqslant M}  \frac {x_{m_1} x_{m_2}     } {\sqrt{m_1 m_2}   }  \SC_2^{\delta} (m_1, m_2) .
\end{align*}
Assume  that $ M < \varPi^{1-\vepsilon}$.  Now we insert the asymptotic formula of $  \SC_2^{\delta} (m_1, m_2)$ in Theorem \ref{thm: C2(m)} and  bound the error term by the sum of
\begin{align*}
    T^{ \vepsilon}  \varPi \sqrt{T}  \mathop{\sum \sum}_{m_1, m_2 \leqslant M}  \frac { 1    } {\sqrt{m_1 m_2}   }     \Lt T^{\vepsilon} M \varPi \sqrt{T},  
\end{align*} 
and
\begin{align*}
	  T^{\vepsilon} \bigg(T + \frac {\varPi^3} {T} \bigg) \mathop{\sum \sum}_{m_1, m_2 \leqslant M}  \frac {   (m_1, m_2)  } { {m_1 m_2}   }   \Lt  T^{\vepsilon} \bigg(T + \frac {\varPi^3} {T} \bigg) . 
\end{align*} 
Then we arrive at the asymptotic formula for $\SM_2^{\delta} $ in the next lemma. 

\begin{lem}\label{lem: M2}
	Let $\varPi = T^{ \vnu }$    and  $M = T^{\varDelta}$.  Then for $\vnu < 1$ and  $\varDelta < \min  \{ 1/2, \vnu   \}  $ there exists $\vepsilon = \vepsilon (\vnu, \varDelta ) > 0$ such that 
	\begin{align}\label{9eq: M2}
		\SM_2^{\delta} = \frac { \varPi T} {\pi \sqrt{\pi}}   {\SM}_{20}^{\delta}    +  O \big(\varPi T^{1-\vepsilon} \big), 
	\end{align}
	with 
	\begin{align}\label{9eq: M20}
		{\SM}_{20}^{\delta}  = \sum_{m} \mathop{\sum \sum}_{(m_1^{\star}, m_2^{\star}) = 1}  \frac { x_{m_1^{\star} m} x_{m_2^{\star} m}   } { m_1^{\star} m_2^{\star} m } \lp \sigma_{-1} (m) \log \frac { {T_{\delta}} } {  {m_1^{\star} m_2^{\star}}  } - 2 \sigma_{-1}'  (m) \rp , 
	\end{align} 
	where $T_{\delta}$ is a multiple of $T$ defined by 
	\begin{align}
		\log T_{\delta} =    \log T + \gamma_{\delta}. 
	\end{align} 
\end{lem}

By removing the co-primality condition $(m_1^{\star}, m_2^{\star}) = 1$ in \eqref{9eq: M20} by the M\"obius function, we have 
\begin{align*}
	 {\SM}_{20}^{\delta} =    \sum_{m} \sum_{d}  \mathop{\sum \sum}_{n_1, \, n_2}  \frac {\mu (d)  x_{dmn_1} x_{dm n_2}   } { d^2 m n_1 n_2 }  \lp  \sigma_{-1} (m) \log \frac {  T_{\delta} } {d^2   {n_1 n_2}  } - 2 \sigma_{-1}'  (m)  \rp. 
\end{align*}
By the definitions of $\sigma_{-1} (m)$ and $\sigma_{-1}' (m)$ in \eqref{1eq: Sigma}, some simple calculations yield
\begin{align*}
	\sum_{d m = h} \frac{   \mu(d) \sigma_{-1} (m)}{d} = 1,   \qquad \sum_{d m = h} \frac{   \mu(d)    \sigma_{-1} (m )}{d} \log d + \sum_{d m = h} \frac{   \mu(d)    \sigma_{-1}' (m )}{d} = 0. 
\end{align*}
Consequently, 
\begin{align}\label{9eq: M20, simplified}
	{\SM}_{20}^{\delta} = \sum_{h}  \mathop{\sum \sum}_{n_1, \, n_2}   \frac {x_{h n_1} x_{h n_2}} {h n_1 n_2 } \log \frac {  T_{\delta} } {  {n_1 n_2}  }. 
\end{align}
Thus if we set
\begin{align}\label{9eq: ym}
	y_{h} = \sum_{n}  \frac {x_{h n}} {n}  , \qquad 	\breve{y}_h = \sum_{n}  \frac {x_{h n} } {n} \log n ,
\end{align} 
then \eqref{9eq: M20, simplified} reads 
\begin{align}\label{9eq: M20=MMM}
	{\SM}_{20}^{\delta} = \log T_{\delta} \cdot  {\SM}_{20}  - 2 \breve{\SM\ }{\!\!}_{20}    , \qquad {\SM}_{20} = \sum_{h} \frac { y_h^2} {h}   , \quad \breve{\SM\ }{\!\!}_{20}  = \sum_{h} \frac {  y_h \breve{y}_h} {h}. 
\end{align} 
By M\"obius inversion, it follows from \eqref{9eq: ym} that 
\begin{align}\label{9eq: xn=y}
	 x_n = \sum_{h} \frac { \mu (h)} {h} y_{h n},   \qquad 
	\breve{y}_h = \sum_{n} \frac {  \varLambda(n)} {n} y_{h n}, 
\end{align}
where  $\varLambda (n) $ is the von Mangoldt function. Note that the assumption $x_1 = 1$ in \eqref{9eq: x1=1} now becomes the linear constraint
\begin{align}\label{9eq: x1=1, 2}
	\sum_{h} \frac {  \mu (h)} {h} y_{h} = 1. 
\end{align}
By Cauchy, the optimal choice of $ y_{h}  $  which minimizes the first quadratic form $ {\SM}_{20} $ in \eqref{9eq: M20=MMM} with respect to \eqref{9eq: x1=1, 2} is given by 
\begin{align}\label{9eq: choice ym}
	y_h =  \frac {\mu (h)} {\Xi (M) }        , \qquad \Xi (M) =    \sum_{h \shskip \leqslant M} \frac {    \mu (h)^2 } {h} .
\end{align}      
We have 
\begin{align}\label{9eq: Xi, asymp}
\Xi (M) =     \frac {6} {\pi^2}   \log M 	+ O(1). 
\end{align}
Moreover, by  \eqref{9eq: xn=y} and \eqref{9eq: choice ym}, it is clear that our  $x_{m}$  are indeed bounded. 
 
By definition, 
\begin{align}\label{9eq: hM20}
	\SM_{20} = \frac 1 {\Xi (M)} . 
\end{align}
By the Prime Number Theorem (see \cite[(3.14.3)]{Titchmarsh-Riemann}), 
\begin{align}\label{9eq: PNT}
 \sum_{ p \shskip \leqslant X} \frac {   \log p } {p} = \log X + O (1),
\end{align} 
and, for $h$ square-free,  
\begin{align}\label{9eq: PNT, 2}
	\sum_{p | h} \frac {\log p} {p} = O    (\log \log (3h)).
\end{align}
It follows from   \eqref{9eq: xn=y}, \eqref{9eq: choice ym}, \eqref{9eq: PNT}, and \eqref{9eq: PNT, 2} that 
\begin{align*}
	  {\breve{y}_h}  /{y_h} = - 
	\mathop{\sum_{p  \nmid h  }}_{  p \shskip \leqslant M/ h  } \frac { \log p} {  p }  
	  =   -      \log  ( {M} /{h}) + O (\log \log (3 h))  . 
\end{align*}
Consequently,
\begin{align*}
	\begin{aligned}
		\breve{\SM\ }{\!\!}_{20} & = \sum_{h \leqslant M} \frac { y_h^2 } {h}      \log h - (\log M - O (\log \log M) ) \sum_{h \leqslant M} \frac {y_h^2} {h}       , 
	\end{aligned}
\end{align*}
and a partial summation using \eqref{9eq: Xi, asymp} yields
\begin{align} \label{9eq: bM20}
		\breve{\SM\ }{\!\!}_{20}   = -  \frac { \log M} {  2 \Xi (M)  } + O  \bigg( \frac {\log \log M} { \log M } \bigg) . 
\end{align}

By \eqref{9eq: M2}, \eqref{9eq: M20=MMM}, \eqref{9eq: Xi, asymp}, \eqref{9eq: hM20}, and \eqref{9eq: bM20}, we conclude with the following asymptotic formula for $\SM_2^{\delta}$. 

\begin{lem}\label{lem: M2, 2}
	 Let $\varPi = T^{ \vnu  }$    and  $M = T^{\varDelta}$.  Then for $\varDelta < \min \{ 1/2, \vnu   \}  $  we have
	 \begin{align}
	 	\SM_2^{\delta} = \frac {  \varPi T} {\pi \sqrt{\pi}}  \cdot \frac   {\pi^2} 6  \frac {1+\varDelta}  {  \varDelta}  \bigg( 1 + O  \bigg( \frac {\log\log T} {\log T}   \bigg) \bigg)   . 
	 \end{align}
\end{lem}

\subsection{Conclusion} \label{sec: conclusion, non-vanishing}
For  $0 < \vnu < 1$  let   $\varPi = T^{ \vnu  }$.   
By Lemmas \ref{lem: M1} and \ref{lem: M2, 2}, it follows from Cauchy   in \eqref{9eq: Cauchy}  that
\begin{align}
	\sumd_{  L (s_f,   f    ) \neq 0} \omega_f \upvarphi (t_f)   \geqslant \frac {  \varPi T} {\pi \sqrt{\pi}}   \bigg( \frac 6 {\pi^2}  \min \bigg\{ \frac 1 3, \frac { \vnu  } { \vnu + 1 }  \bigg\} - \vepsilon  \bigg) , 
\end{align} 
 as $T \ra \infty$  for any $\vepsilon > 0$.
Next, we apply Lemma \ref{lem: unsmooth, 1} with $a_n =   \delta (L (s_f,   f    ) \neq 0)$ (namely, the Kronecker $\delta$ symbol) and $H = \varPi^{1+\vepsilon}$, obtaining 
\begin{align}\label{9eq: unsmoothed}
		\mathop{\sumd_{ |t_f - T| \shskip \leqslant H}}_{L (s_f,   f    ) \neq 0}  \omega_f     \geqslant \frac {2 H T} {\pi^2 }   \bigg( \frac 6 {\pi^2}   \min \bigg\{ \frac 1 3, \frac { \mu  } { \mu + 1 }  \bigg\} - \vepsilon  \bigg) ,\qquad \mu = \frac {\log H} {\log T}, 
\end{align} 
for any $ T^{ \vepsilon} \leqslant  H \leqslant T/3$. 
 Recall the Weyl law (see \eqref{2eq: Weyl, 1}):
\begin{align}\label{9eq: Weyl, 1}
	\sumd_{  t_f \leqslant T} \omega_f = \frac {T^2} {2 \pi^2} + O (T);
\end{align} 
 as a consequence, 
\begin{align}\label{9eq: Weyl, 2}
	 {\sumd_{  |t_f - T| \shskip \leqslant H}}  \omega_f     = \frac {2 H T} {\pi^2 } + O (T) . 
\end{align} 
From \eqref{9eq: unsmoothed}--\eqref{9eq: Weyl, 2} we deduce the following harmonic weighted non-vanishing results.

\begin{theorem}
Let $\vepsilon > 0$ be arbitrarily small. As $T \ra \infty$, we have  
	\begin{align}\label{9eq: case t<T}
		\mathop{\sumd_{  t_f \leqslant T}}_{L (s_f,   f    ) \neq 0}  \omega_f  \geqslant   \bigg(\frac 2 {\pi^2}   - \vepsilon \bigg) \sumd_{ t_f \leqslant T} \omega_f , 
	\end{align}
	and, for $0 < \mu < 1$,  
	\begin{align}\label{9eq: case t-T<H}
		\mathop{\sumd_{  |t_f - T| \shskip \leqslant T^{\mu}}}_{L (s_f,   f    ) \neq 0}  \omega_f  \geqslant   \bigg( \frac {6} {\pi^2 } \min \bigg\{ \frac 1 3, \frac { \mu  } { \mu + 1 }  \bigg\} - \vepsilon  \bigg)   \sumd_{  |t_f - T| \shskip \leqslant T^{\mu}} \omega_f , 
	\end{align}
where the sums are over $\SB_{\delta}$. 
\end{theorem}

Note that \eqref{9eq: case t<T} follows from \eqref{9eq: unsmoothed} and \eqref{9eq: Weyl, 1} with $ H = T/3 $ along with a dyadic summation, while \eqref{9eq: case t-T<H} follows directly from \eqref{9eq: unsmoothed} and \eqref{9eq: Weyl, 2}.    

Since the harmonic weight 
\begin{align*}
	\omega_f = \frac {2} {  L(1, \mathrm{Sym}^2 f    )}, 
\end{align*}
one may remove $\omega_f$ without compromising the strength of results by the method of Kowalski and Michel \cite{KM-Analytic-Rank}, adapted for the Maass-form case  in \cite{BHS-Maass} (see also the paragraph below (2.9) in \cite{IS-Siegel}).  Thus after the removal of  $\omega_f$  we obtain the non-vanishing results in Theorems \ref{thm: non-vanishing} and \ref{thm: non-vanishing, short}.

\section{Moments with no Twist and Smooth Weight}\label{sec: unsmooth}

\begin{theorem}
	Let $ T^{ \vepsilon} \leqslant H \leqslant T/3$. Then 
	\begin{align}\label{10eq: 1st moment}
		\sumd_{ |t_f - T| \leqslant H }   \omega_f L (s_f,    f     ) =  \frac { 2  } {\pi^2} H T   + O \big(T^{ \vepsilon} (  H \sqrt{T} + T )  \big). 
	\end{align}
\end{theorem}

\begin{proof}
By applying Lemma \ref{lem: unsmooth, 1} to the first asymptotic formula in Corollary \ref{cor: 1st moment}, we have
\begin{align*}
	 \sumd_{ |t_f - T| \leqslant H }   \omega_f L (s_f,    f     ) = \frac { 2   } {\pi^2} H T + O \big(T^{\vepsilon } (H \sqrt{T} + \varPi T) \big), 
\end{align*}
for any $T^{\vepsilon} \leqslant \varPi^{1+\vepsilon} \leqslant   H \leqslant T / 3 $, and we obtain \eqref{10eq: 1st moment} on the choice $ \varPi = H / \sqrt{T} + T^{\vepsilon}   $. 
\end{proof}

\begin{theorem}
	 Let $ T^{ \vepsilon} \leqslant H \leqslant T/3$. Then
	 \begin{align}\label{11eq: 2nd moment}
	 	 \sumd_{ |t_f - T| \leqslant H }   \omega_f |L (s_f,    f     )|^2  = \frac {1} {\pi^2} \! \int_{T-H}^{T+H} K  (  \log K + \gamma_{\delta}) \nd K + O \big(T^{ \vepsilon} ( {\textstyle H \sqrt{T} + \sqrt{H} T} )  \big). 
	 \end{align}
\end{theorem}

\begin{proof}
	By applying Lemma \ref{lem: unsmooth, 1}  to the second asymptotic formula in Corollary \ref{cor: 1st moment}, we infer that the spectral second moment on the left of  \eqref{11eq: 2nd moment} equals 
	\begin{align*}
	 \frac {1} {\pi^2} \int_{T-H}^{T+H} K  (  \log K + \gamma_{\delta}) \nd K + O \bigg(T^{\vepsilon} \bigg(H\sqrt{T} + \frac {H T } {\varPi  } + \varPi T \bigg)\bigg), 
	\end{align*}
for any $\varPi^{1+\vepsilon} \leqslant H \leqslant T/3$, and we obtain \eqref{11eq: 2nd moment} on the choice $\varPi =   \sqrt{H}$. 
	\end{proof}

Corollary \ref{cor: unsmoothed} is a direct consequence of the above asymptotics with $H = T/3$ along with a dyadic summation.

{\large \part{Density Theorems}  }

\vspace{5pt}

\section{Sums of Variant Kloosterman Sums} 

Let us introduce the variant Kloosterman sum
\begin{align} \label{2eq: defn V}
	V_{q} (m, n; c) =	\mathop{\sum_{\valpha     (\mathrm{mod}\, c)}}_{ (\valpha     (q-\valpha    ), c) = 1 }  e \bigg(   \frac {  \widebar{\valpha    } m +  \overline{(q - \valpha)    } n } {c} \bigg). 
\end{align} The definition of $V_{q} (m, n; c)$ is essentially from Iwaniec--Li  \cite[(2.17)]{Iwaniec-Li-Ortho} and arose in the identity of Luo \cite[\S 3]{Luo-Twisted-LS}: 
\begin{align}\label{2eq: S = V}
	S (m, n; c)  e \Big(\frac {m+n} {c} \Big) = \sum_{r q   = c } V_{q} (m, n; r). 
\end{align} 
The purpose of this section is to prove the following analogue of \cite[Lemma 6.1]{ILS-LLZ}.

\begin{lem}
	Assuming  the Riemann hypothesis for Dirichlet $L$-functions, for any $x \geqslant 2$, we have 
	\begin{align}\label{12eq: sum of Vq}
		\mathop{\sum_{p \leqslant x} }_{p \nmid c} V_{q} (m p, n ; c) \log p = \frac {x} {\varphi (c)} R (m; c) R_q  (n; c) + O (\varphi (c) \sqrt{x} \log^2 (cx) ) , 
	\end{align} 
 where $\varphi (c) $ is the Euler function, $R (m; c) = S (m, 0; c)$, and $R_q (n; c) = V_q (0, n; c) $. 
\end{lem}

\begin{proof}
	As in  \cite[\S 6]{ILS-LLZ}, we shall use the Riemann hypothesis for Dirichlet $L$-functions in the form
	\begin{align*}
		\sum_{p \leqslant x} \vchi (p) \log p =  \delta ({\vchi}) \cdot  x + O (\sqrt{x} \log^2 (c x)), 
	\end{align*}
where $ \vchi $ is any character to modulus $c$ and $\delta ({\vchi})$ is the indicator of the principal character. 
By definition, 
\begin{align*}
	\sum_{\valpha     (\mathrm{mod}\, c)} \vchi (\valpha) V_q (\valpha m, n; c) = \mathop{\mathop{\sum \sum}_{\valpha, \beta     (\mathrm{mod}\, c)}}_{(\beta(q -\beta), c) = 1} \vchi (\valpha) e \bigg(   \frac { \valpha \widebar{\beta    } m +  \overline{(q - \beta )   } n } {c} \bigg) , 
\end{align*}
and on the change $ \valpha \ra \valpha \beta $ it splits into the product $ G  (m; \vchi) G_{  q} (n; \vchi)$ with
\begin{align*}
	G  (m; \vchi) = \sum_{\valpha     (\mathrm{mod}\, c)} \vchi (\valpha) e \Big(   \frac { \valpha   m  } {c} \Big) , \qquad G_{  q}  (n; \vchi) =   \mathop{\sum_{\valpha     (\mathrm{mod}\, c)}}_{ (      q-\valpha    , c) = 1 } \vchi (\valpha) e \bigg(   \frac {     \overline{(q - \valpha)    } n } {c} \bigg). 
\end{align*}
Note that for the principal character these (Gauss) sums become the (Ramanujan) sums $R (m; c)$ and $R_q (n; c)$, and that by the orthogonality of characters 
\begin{align*}
	\sum_{\vchi (\mathrm{mod}\, c)} |G  (m; \vchi)|^2 = R (0; c)^2 , \qquad \sum_{\vchi (\mathrm{mod}\, c)} |G_{ q} (m; \vchi)|^2 = R (0; c) R_q (0; c) . 
\end{align*} 
Finally, from these formulae we derive
\begin{align*}
	\mathop{\sum_{p \leqslant x} }_{p \nmid c} V_{q} (m p, n ; c) \log p & = \frac 1 {\varphi (c)} \sum_{\vchi (\mathrm{mod}\, c)} \bigg( \sum_{\valpha     (\mathrm{mod}\, c)} \vchi (\valpha) V_q (\valpha m, n; c) \bigg) \bigg( \sum_{p \leqslant x} \overline{\vchi} (p) \log p  \bigg)\\
	& = \frac 1 {\varphi (c)} \sum_{\vchi (\mathrm{mod}\, c)} G  (m; \vchi) G_{  q} (n; \vchi) \big( \delta {(\vchi)} \cdot x + O (\sqrt{x} \log^2 (c x))  \big),
\end{align*}
and hence \eqref{12eq: sum of Vq}. 
\end{proof}

Later in our application,  $m = 1$ and $n = \pm 1$. It is well-known that the Ramanujan sum 
\begin{align}\label{12eq: R(1;c)}
	R (1; c) = \mu (c). 
\end{align}
By the Chinese remainder theorem, it is easy to prove 
\begin{align}\label{12eq: Rq(1;c)}
	R_q (\pm 1; c) = O \big( \tau (c) \mu (c)^2  \big). 
\end{align}

\section{Further Analysis for the Bessel Integrals} \label{sec: Bessel, 2}

Set $\upgamma (t) \equiv 1$ in \eqref{4eq: h2(t;y)} so that
\begin{align*}
	h_2 (t; y) =  2   \cos (2 t \log y) (\upvarphi (t) + \upvarphi (-t)). 
\end{align*}
According to \eqref{4eq: H2=H} and \eqref{4eq: H(x,y) = I(v,w)}, up to a negligible error, we split $H_2^{\pm} (x, y) $ into four similar parts, one of which reads
\begin{align}\label{13eq: I-+(v,w)}
   \frac {\exp (- i (v \pm w) ) \cdot I^{_\pm}_{^{- +} }  (v, w)} 2    , \qquad   v = \frac {xy} {2}, \quad w = \frac {x/y} 2, 
\end{align}
with
\begin{align}\label{13eq: I(v,w)}
	I^{_\pm}_{^{- +} } (v, w) = \varPi T  \int_{-\varPi^{\vepsilon}/\varPi}^{\varPi^{\vepsilon}/ \varPi} 
	g (  \varPi r) \exp ( 2 i T r -  i (\uprho(r) v \pm \uprho (-r) w ) )   \nd r ,
\end{align} 
for $ \uprho (r) = \exp (r) - 1. $ More explicitly, the pair of $\pm$ in the subscript of $I^{_\pm}_{^{\pm \pm} } (v, w) $ is determined by that in  $ \pm i \uprho(\pm r) v $ contained in the integral.   Later, we shall have
\begin{align*}
	x = \frac {4\pi \sqrt{p}} {c}, \qquad y = \sqrt{p}, 
\end{align*}
and
\begin{align*}
	v = \frac {2\pi p} {c}, \qquad w = \frac {2\pi} {c},
\end{align*}
so  in practice $w $ is small and inessential, and will be suppressed from the notation. 

\begin{lem}\label{lem: further analysis}
	 Write 
	 \begin{align}\label{13eq: derivative I(v,w)}
	 	I^{_{\pm \natural}}_{^{- +} } (v, w) = \frac { \partial I^{_\pm}_{^{-+} } (v, w)  } {\partial v}  . 
	 \end{align}
Then we have bounds
\begin{align}\label{13eq: bounds for I(v,w)}
	I^{_\pm}_{^{-+} } (v, w) \Lt T^{1+\vepsilon}, \qquad I^{_{\pm \natural}}_{^{-+} } (v, w)  \Lt T^{1+\vepsilon}/ \varPi. 
\end{align}
Moreover, if $ w < \varPi$, then   $I^{_\pm}_{^{-+} } (v, w),  I^{_{\pm \natural}}_{^{-+} } (v, w) = O (T^{-A})$ unless $|2T - v| \Lt \varPi^{\vepsilon} ( \varPi +T / \varPi )$. 
\end{lem}

\begin{proof}
	Note that  $I^{_{\pm \natural}}_{^{-+} } (v, w) $ is very similar to $I^{_{\pm }}_{^{-+} } (v, w) $:
	\begin{align} \label{13eq: In(v,w)}
		I^{_{\pm \natural}}_{^{-+} } (v, w) = - i \varPi T  \int_{-\varPi^{\vepsilon}/\varPi}^{\varPi^{\vepsilon}/ \varPi} 
		g (  \varPi r) \cdot   \uprho (r) \exp ( 2 i T r -  i (\uprho(r) v \pm \uprho (-r) w ) )   \nd r .
	\end{align}
Clearly \eqref{13eq: bounds for I(v,w)} follows from \eqref{13eq: I(v,w)} and \eqref{13eq: In(v,w)} by trivial estimation. As for the second statement, we use the method of stationary phase. By $w < \varPi$, we may as well absorb the factor $ \exp (\mp i \uprho (-r)w)$ into the weight $g (\varPi r)$ so that the phase function is $ 2 T r - \uprho (r) v $, while its derivative equals
\begin{align*}
	2 T -  \exp(r) v = 2 T - v + O ( v \varPi^{\vepsilon}/ \varPi ) .  
\end{align*} 
Thus, in the case $ T \Gt  |2T - v| \Gt \varPi^{\vepsilon} ( \varPi +T / \varPi ) $, by an application of Lemma B.1 in \cite{Qi-GL(2)xG(2)-RS} with $S = 1$, $ P = 1/\varPi $, $Q = 1$, $Z = T $, and $ R = \varPi^{\vepsilon} ( \varPi +T / \varPi ) $, we infer that the integrals in \eqref{13eq: I(v,w)} and \eqref{13eq: In(v,w)} are negligibly small.  Note here that $R $ is larger than both $ \varPi^{1+\vepsilon}$ and $  T^{1/2+\vepsilon}$.  However, for $ |2 T - v | \Gt T $, the integrals are always negligibly small if we  apply Lemma B.1 in \cite{Qi-GL(2)xG(2)-RS} with  $S = 1$, $ P = 1/\varPi $, $Q = 1$, $Z =   v$, and $ R = v + T $.  
\end{proof}

\begin{remark}\label{rem: Ipm}
As for the other three integrals $I^{_\pm}_{^{+-} } (v, w) $, $I^{_\pm}_{^{+ +} } (v, w) $, and $I^{_\pm}_{^{- -} } (v, w) $, the first one is similar, while the latter two are always negligible. 
\end{remark}

\section{The Explicit Formula}


By applying \cite[Proposition 2.1]{Rudnick-Sarnak-RMT} with 
\begin{align*}
	g (u) = \frac 1 {\log T} {e \lp -  \frac { u t_f  } {2\pi}  \rp}   \hat{\upphi} \bigg( \frac {u} {\log T}\bigg), \qquad h (r) = \upphi \bigg(\frac {\log T} {2\pi} (r-t_f) \bigg), 
\end{align*}
we obtain the explicit formula
\begin{align}\label{12eq: D=H-P}
D  (f; \upphi; T) = H (f; \upphi; T) - P (f; \upphi; T), 
\end{align}
with
\begin{align}
D  (f; \upphi; T) =	\sum_{\gamma_f }  \upphi \bigg(\frac {\log T} {2\pi} (\gamma_f-t_f) \bigg) , 
\end{align}
 as defined before in \eqref{1eq: defn D(f)}, and
\begin{align}\label{12eq: H(f;T)}
H ( f; \upphi; T) =	\frac 1 {\pi}  \!  \int_{-\infty}^{\infty} \! \!  \upphi \bigg(\frac {\log T} {2\pi} r \bigg) \bigg\{  \frac {\Gamma_{\mathbf{R}}'  } {\Gamma_{\mathbf{R}}  } \big(2\vkappa_f +  i r  \big) + \mathrm{Re}\, \frac {\Gamma_{\mathbf{R}}'  } {\Gamma_{\mathbf{R}}  } \big(2\vkappa_f + 2 it_f + ir   \big) \bigg\} \nd r , 
\end{align}
\begin{align}\label{12eq: P(f;T)}
	P (f; \upphi; T) =    2  \sum_{n=1}^{\infty} \frac {\Lambda (n) a_f (n) \cos (t_f \log n)} { \sqrt{n} \, {\log T} }        \hat{\upphi} \bigg( \frac {\log n} {\log T}\bigg), 
\end{align}
where
\begin{align*}
	\vkappa_f = \frac {1+2\delta_f} {4} , 
\end{align*} (see \eqref{3eq: FE, Q} and \eqref{3eq: GammaR(s)}),  as usual $\Gamma_{\mathbf{R}} (s) = \pi^{-s/2} \Gamma (s/2) $, $\Lambda (n) $ is the  von Mangoldt function ($\Lambda (n) = \log p$ if $n = p^{\vnu}$ and $0$ otherwise), and 
\begin{align*}
	a_f (p^{\vnu} ) = \valpha_f (p)^{\vnu} + \beta_f (p)^{\vnu}, 
\end{align*}
for $\valpha_f (p)$ and $ \beta_f (p)$ defined as in the Euler product \eqref{3eq: Euler}.

\begin{lem}\label{lem: asymp for H}
	Let $\mathrm{Re} (s) \geqslant \delta > 0$.   Then for $T, |s|  \Gt 1$, we have 
	\begin{align}\label{12eq: gamma integral}
		\frac 1 {\pi}   \int_{-\infty}^{\infty}   \upphi \bigg(\frac {\log T} {2\pi} r \bigg)   \frac {\Gamma_{\mathbf{R}}'  } {\Gamma_{\mathbf{R}}  }  ( 2 s +  i r   )   \nd r =   \frac {\log (s/\pi)} {\log T}  \hat{\upphi} (0) + O_{\delta} \bigg(\frac 1 {  \log T} \bigg). 
	\end{align}
\end{lem}

\begin{proof}
By definition, $\Gamma_{\mathbf{R}} (s) = \pi^{-s/2} \Gamma (s/2) $, and by the change $r \ra 2\pi r/ \log T$,  the integral in \eqref{12eq: gamma integral} is turned into 
\begin{align*}
\frac {1} {\log T} \int_{-\infty}^{\infty}   \upphi (r)  \lp  \frac {\Gamma'} {\Gamma }  \bigg( s +  \frac { \pi i r } {\log T}  \bigg) -\log \pi \rp    \nd r . 
\end{align*}
It follows from the Stirling formula 
	\begin{align*}
		\frac {\Gamma'(s)} {\Gamma (s) }= \log s + O \lp \frac 1 {|s|} \rp, 
	\end{align*}
and the crude uniform bound
\begin{align*}
	\log (s + w) - \log s = O_{\delta} \lp \left|\frac {w} {s} \right| \log (|w|+2) \rp, \qquad \text{($\mathrm{Re}(s) \geqslant \delta > 0, \, \mathrm{Re}(w) \geqslant 0$)} , 
\end{align*} 
that the integral above equals
\begin{align*}
& \quad \  \frac {1} {\log T} \int_{-\infty}^{\infty}   \upphi (r)      \log  \bigg( \frac s {\pi} +  \frac {   i r } {\log T}  \bigg)   \nd r + O \bigg(\frac 1 {  \log T} \bigg) \\
& = \frac {1} {\log T} \int_{-\infty}^{\infty}   \upphi (r) \bigg(\log \lp \frac s {\pi} \rp  + O \bigg(\frac {|r| \log (|r|+2)} {|s| \log T} \bigg)   \bigg)  \nd r + O \bigg(\frac 1 { \log T} \bigg) \\
& = \frac {\log (s/\pi)  } {\log T} \int_{-\infty}^{\infty}   \upphi (r)   \nd r + O \bigg(\frac 1 {  \log T} \bigg),  
\end{align*} 
as desired. 
\end{proof}

\begin{remark}
	By modifying the proof of Lemma {\rm\ref{lem: asymp for H}}, one may improve  the error term $O_{\delta} (1/\log T)$   into $ O_{\delta, \vepsilon} (1 / |s| \log T) $ if $T^{\vepsilon} \Lt |s| \Lt T^{1/\vepsilon} $. 
\end{remark}

By applying Lemma \ref{lem: asymp for H} with $s = \vkappa_f$ or $s = \vkappa_f + it_f$ to \eqref{12eq: H(f;T)}, we have 
\begin{align}\label{12eq: H}
H ( f; \upphi; T) =	\frac {\log |\vkappa_f + it_f|  } {\log T} \hat{\upphi} (0) + O \bigg(\frac 1 {  \log T} \bigg).
\end{align}


\begin{lem}\label{lem: sum log p/ ps}
Assume the Riemann hypothesis {\rm(}for the Riemann zeta function{\rm)}.	Let $ s =  1 + it $ and $t \Gt 1$. Then  
\begin{align}\label{11eq: sum log p}
	\sum_{p < x} \frac {\log p} {p^s} = O(  \log \log t),
\end{align}
for $ x < t$, where the implied constant is absolute. 
\end{lem}

\begin{proof}
As usual, write $s = \sigma + it$. By \cite[\S \S 14.5, 14.33]{Titchmarsh-Riemann}, under the Riemann hypothesis, we have
\begin{align}\label{11eq: zeta'/zeta}
	\frac {\zeta' (s)} {\zeta (s)}  = O  \big( ({\log t})^{2-2\sigma} \big), 
\end{align}
uniformly for $1/2 < \sigma_0 \leqslant \sigma \leqslant \sigma_1 < 1$, while 
\begin{align}\label{11eq: zeta'/zeta (1+it)}
	\frac {\zeta' (s)} {\zeta (s)} = O (\log \log t), 
\end{align}
for $ \sigma = 1$. 
By  \eqref{11eq: zeta'/zeta}, it follows from the Perron formula in \cite[Lemma 3.12]{Titchmarsh-Riemann} and the contour-shift argument in  \cite[\S \S 3.14, 3.15]{Titchmarsh-Riemann} with $c = 1/ \log x$ and $\delta = 1 - \sigma_0 $ that
\begin{align*}
	 \sum_{n < x}  \frac {\Lambda (n)} {n^{s}} + \frac {\zeta' (s)} {\zeta (s)} - \frac {x^{1-s}} {1-s} =  O \bigg( \frac {\log^2 x} {T} + \frac {\log T \cdot \log^{ 2\delta} t} {x^{\delta}}  + \frac {\log^{2 \delta} t} {T}  \bigg) ,  
\end{align*}
in the case that $\sigma = 1$ and $T < t/2$; here one may let $x$ be  half-integral. 
For $x < t$, on choosing $\sqrt{T} = \log x + \log^{\delta} t$, we obtain 
\begin{align} \label{11eq: log p bound}
	\sum_{p < x}  \frac {\log p} {p^{s}} + \frac {\zeta' (s)} {\zeta (s)}  =  O \bigg( 1 + \frac {\log \log t \cdot \log^{ 2\delta} t} {x^{\delta}}    \bigg). 
\end{align} 
Note here that 
\begin{align*}
	\sum_{n < x}  \frac {\Lambda (n)} {n^{s}} = \sum_{p < x} \frac {\log p} {p^s} + O (1).   
\end{align*} 
Also, trivially we have
\begin{align} \label{11eq: trivial log x}
	\sum_{p < x}  \frac {\log p} {p^{s}} = O (\log x).  
\end{align} 
According as $x > \log^2 t$ or not, the bound \eqref{11eq: sum log p} follows from \eqref{11eq: zeta'/zeta (1+it)}, \eqref{11eq: log p bound} or \eqref{11eq: trivial log x}.  
\end{proof}

Now we turn to the sum $P ( f; \upphi; T)$ as defined in \eqref{12eq: P(f;T)}: unlike the cases of central $L$-values in \cite{ILS-LLZ,Alpoge-Miller-1,Qi-Liu-LLZ}, this will not contribute a main term. Suppose that $\hat{\upphi}$ has support in $(-2, 2)$, say, and  $ t_f $ is close to $T$.  

Firstly, we may discard those terms with $n = p^{\vnu}$ for $\vnu \geqslant 3$, since  $|a_f (p^{\vnu}) | \leqslant 2 p^{\frac 7 {64} \vnu}$ by the Kim--Sarnak bound \eqref{3eq: Kim-Sarnk}, and hence their contribution is $O (1/\log T)$. Next, by partial summation, it follows from Lemma \ref{lem: sum log p/ ps} that
\begin{align} \label{12eq: P}
	\sum_{p} \frac { \cos (2 t_f \log p )  \log p} {p  \log T}  \hat{\upphi}  \bigg(\frac {2 \log p} {\log T} \bigg) = O_{\upphi}  \bigg( \frac {\log \log T} {\log T}  \bigg)  . 
\end{align} 
Since $a_f (p) = \lambda_{f} (p)$ and $ a_f (p^2) = \lambda_{f} (p^2) - 1 $, we conclude that   
\begin{align} \label{12eq: P1, P2}
	 P ( f; \upphi; T) = P_1 (f; \upphi; T) + P_2 (f; \upphi; T) + O_{\upphi}  \bigg( \frac {\log \log T} {\log T}  \bigg) ,
\end{align}
with 
\begin{align}\label{12eq: P12}
	P_{\vnu} (f; \upphi; T) = 2 \sum_{p}  \lambda_f (p^{\vnu})  \frac {\cos (\vnu t_f \log p ) \log p} {p^{\vnu/2} \log T}    \hat{\upphi}  \bigg(\frac {\vnu \log p} {\log T} \bigg). 
\end{align}

Finally, we obtain the following asymptotic formula by combining \eqref{12eq: D=H-P}, \eqref{12eq: H},  and  \eqref{12eq: P1, P2}. 

\begin{lem}\label{lem: explicit formula}
	We have
	\begin{align}\label{12eq: asymp D}
		D (f; \upphi; T) \!= \! \frac {\log |\vkappa_f + it_f|  } {\log T} \hat{\upphi} (0) - P_1 (f; \upphi; T) - P_2 (f; \upphi; T) + O_{\upphi} \! \bigg( \frac {\log \log T} {\log T}  \bigg) \!  ,
	\end{align}
 if $\hat{\upphi}$ has support in $(-2, 2)$. 
\end{lem}

\section{Proof of the Density Theorems} 

As in \eqref{4eq: defn phi} let us define 
\begin{align}\label{13eq: phi}
	 	\upvarphi  (t)  =  \upvarphi_{T, \varPi} (t) =  \exp \left( - \frac {(t-T)^2}  {\varPi^2} \right),
\end{align}
and consider the $\SB_{\delta} $-averaged $1$-level density  
\begin{align}\label{12eq: density}
	\SD^{\delta}_{\natural}   ( \upphi;  T, \varPi ) =   \sum_{f \in \SB_{\delta} } \omega_f	 D (f; \upphi; T) \big(\upvarphi_{T, \varPi} (t_f) + \upvarphi_{T, \varPi} (- t_f) \big) .
\end{align}
Clearly 
$\SD^{\delta}_{\natural}   ( \upphi;  T, \varPi ) $ differs from  $\SD^{\delta}   ( \upphi;  T, \varPi ) $ as in \eqref{1eq: density} only by  an exponentially small error term. 
By inserting \eqref{12eq: asymp D}, we have 
\begin{align}
\SD^{\delta}_{\natural}   ( \upphi;  T, \varPi ) \! =\! \frac {\varPi T } {\pi \sqrt{\pi}} \hat{\upphi} (0)  - \SP^{\delta}_{1}   ( \upphi;  T, \varPi ) -  \SP^{\delta}_{2}   ( \upphi;  T, \varPi ) + O_{\upphi} \bigg( \frac {\varPi T \log \log T} {\log T} \bigg), 
\end{align}
with 
\begin{align}\label{13eq: P12}
	\SP^{\delta}_{\vnu}   ( \upphi;  T, \varPi ) =  \sum_{f \in \SB_{\delta} } \omega_f	 P_{\vnu} (f; \upphi; T) \big(\upvarphi_{T, \varPi} (t_f) + \upvarphi_{T, \varPi} (- t_f) \big). 
\end{align}
Note here that the main term $ \varPi T \hat{\upphi} (0)  / \pi \sqrt{\pi} $ is due to the Weyl law as in \eqref{2eq: Weyl, 1}. 
Therefore Theorems \ref{thm: density, limited} and \ref{thm: density} are deducible easily from the following bounds for $  \SP^{\delta}_{1}   ( \upphi;  T, \varPi ) $ and $  \SP^{\delta}_{2}   ( \upphi;  T, \varPi ) $. 

\begin{lem}\label{lem: P limited}
 We have  
	 \begin{align}\label{15eq: P(phi), 0}
	 	 \SP^{\delta}_{1}   ( \upphi;  T, \varPi ) = O_{\upphi}  (\varPi\sqrt{T} \log T ) ,  \qquad 
 	\SP^{\delta}_{2}   ( \upphi;  T, \varPi ) = O_{\upphi} (\varPi \log^2 T ) , 
 \end{align}
 if $\hat{\upphi}$ has support in $(-1, 1)$. 
\end{lem}

\begin{lem}\label{lem: P extended}
Let $\varPi = T^{\mu}$ for $0 < \mu < 1$. Assume the Riemann hypothesis for Dirichlet $L$-functions. Define 
\begin{align}\label{15eq: v(mu)}
	v (\mu) = \left\{ \begin{aligned}
		& 1, & & \text{ if } 0 < \mu \leqslant 1/3,\\
		& 3 \mu, & & \text{ if } 1/3 < \mu \leqslant 1/2, \\
		& 1 + \mu, & & \text{ if } 1/2 < \mu < 1.
	\end{aligned}\right.
\end{align} Then we have  
	\begin{align}\label{15eq: P(phi)}
		\SP^{\delta}_{1}   ( \upphi;  T, \varPi ), \, \SP^{\delta}_{2}   ( \upphi;  T, \varPi ) = o (\varPi T) ,
	\end{align}
	if $\hat{\upphi}$ has support in $(- v(\mu), v (\mu))$. 
\end{lem}

Subsequently, we shall focus on $ \SP^{\delta}_{1}   ( \upphi;  T, \varPi )  $ as the treatment of  $ \SP^{\delta}_{2}   ( \upphi;  T, \varPi )  $ is much easier and yields a better bound. 

\subsection*{Notation}  Suppose $\hat{\upphi}$ has support in the closed interval $ [-v, v] $, say for $v < 2$. Let $\varPi = T^{\mu}$ and $P = T^{v}$.  

For brevity, let us write $ \upvarphi (t) = \upvarphi_{T, \varPi} (t)$ as before and  suppress $T, \varPi$ from the notation like $\SP^{\delta}_{1}   ( \upphi;  T, \varPi ) $.

\subsection{Application of the Kuznetsov Formula} \label{sec: Kuz, density} By \eqref{12eq: P12} and \eqref{13eq: P12}, we may write
\begin{align*}
	 \SP^{\delta}_{1}   ( \upphi) = \sum_{p } \frac {\log p} {\sqrt{p} \log T}  \hat{\upphi}  \bigg(\frac { \log p} {\log T} \bigg) 
	 \sum_{f \in \SB_{\delta} } \omega_f \lambda_{f} (p ) h_2 ( t_f ; \sqrt{p}  ), 
\end{align*}
where, as in \eqref{4eq: h2(t;y)}, 
\begin{align*}
	h_2 (t; y) = 2  \cos (2t \log y) (\upvarphi (t) + \upvarphi (-t)). 
\end{align*}
By applying the Kuznetsov  trace formula in \eqref{2eq: Kuznetsov}, we have
\begin{align}\label{15eq: P=E+O}
	2 \SP^{\delta}_{1}   ( \upphi) = 2 (\delta-1)  \SE_1  (\upphi) +  \SO_{+}  (\upphi) + (-1)^{\delta} \SO_{-} (\upphi) , 
\end{align}
where 
\begin{align}\label{13eq: Eis}
	\SE_{1} (\upphi) = \sum_{p  } \frac {\log p} {\sqrt{p} \log T}  \hat{\upphi}  \bigg(\frac { \log p} {\log T} \bigg) \cdot \frac 4 {\pi}    \int_{-\infty}^{\infty} \omega(t)    \cos ( t \log p)   \tau_{it}(p)   \upvarphi (t) \nd t, 
\end{align}
\begin{align}\label{13eq: O+-}
\SO_{\pm}  (\upphi) =	\sum_{p  } \frac {\log p} {\sqrt{p} \log T}  \hat{\upphi}  \bigg(\frac { \log p} {\log T} \bigg) \sum_{c} \frac {S (p, \pm 1; c)} {c} H_2^{\pm}  \bigg( \frac {4\pi \sqrt{p}} {c} , \sqrt{p}   \bigg) . 
\end{align}

First of all, the Eisenstein term $ \SE_{1} (\upphi)  $ is indeed small.  Since the integral in \eqref{13eq: Eis} is trivially $O (\varPi \log^2 T)$ (by \cite[(3.11.10)]{Titchmarsh-Riemann}), we have 
\begin{align}\label{13eq: Eis, bound}
	 \SE_{1} (\upphi) = O (\varPi \sqrt{P} \log T) .
\end{align}
Note that the similar Eisenstein term $ \SE_{2} (\upphi)  $ is even smaller: $ O (\varPi \log^2 T) $. 

It is now left to consider the off-diagonal sums $ \SO_{\pm}  (\upphi)  $ as in \eqref{13eq: O+-}.

\subsection{Proof of Lemma \ref{lem: P limited}} It follows from Corollary \ref{cor: u<T} that  the Bessel integral  $ H_2^{\pm}   (   {4\pi \sqrt{p}} / {c} , \allowbreak  \sqrt{p}    )$ in  \eqref{13eq: O+-} is negligibly small for $ c \Gt p / T$. Therefore   $ \SO_{\pm}  (\upphi)  $  is negligible for $v < 1$ so that $ p / T = o (1) $ for all $p \leqslant P = T^{v}$.  Similarly, the off-diagonal contribution to  $ \SP^{\delta}_{2}   ( \upphi) $ is also negligible for $v < 1$. 

Note that in the odd case  $ \SP^{1}_{1}   ( \upphi)  $ and $ \SP^{1}_{2}   ( \upphi)  $ are actually negligibly small as there is no Eisenstein contribution.

\subsection{Proof of Lemma \ref{lem: P extended}}

In light of \eqref{15eq: P=E+O} and \eqref{13eq: Eis, bound}, it suffices to prove $ \SO_{\pm} (\upphi) = o (\varPi T) $ for any $v < v (\mu)$.   As in \S \ref{sec: Bessel, 2}, let us split $\SO_{\pm}  (\upphi) $ into four similar parts, one of which contains $ I^{_\pm}_{^{- +} } (2\pi p/c, 2\pi / c)  $ as in \eqref{13eq: I-+(v,w)} and \eqref{13eq: I(v,w)} and reads:
\begin{align*}
	\SO_{^{\pm}}^{_{-+}} (\upphi) = \sum_{p  } \frac {\log p} {2\sqrt{p} \log T}  \hat{\upphi}  \bigg(\frac { \log p} {\log T} \bigg) \sum_{c} \frac {S (p, \pm 1; c)} {c} e \lp - \frac {p\pm 1} {c} \rp I^{_\pm}_{^{- +} }  \bigg( \frac {2\pi  {p}} {c} , \frac {2\pi} {c}   \bigg). 
\end{align*}
For brevity, let us  write $ I^{_\pm}_{^{- +} } ( 2\pi p/c ) = I^{_\pm}_{^{- +} } (2\pi p/c, 2\pi / c) $. By the Luo identity \eqref{2eq: S = V}, 
\begin{align*}
	\SO_{^{\pm}}^{_{-+}} (\upphi) = \sum_{p  } \frac {\log p} {2 \sqrt{p} \log T}  \hat{\upphi}  \bigg(\frac { \log p} {\log T} \bigg) \sum_{c} \sum_{q}  \frac {\overline{V_q (p, \pm 1; c)}} {c q}  I^{_\pm}_{^{- +} }  \bigg( \frac {2\pi  {p}} {c q}   \bigg). 
\end{align*}
Next, we change the order of summations and truncate the $c$- and $q$-sums at $c q \Lt P / T$ (again by Corollary \ref{cor: u<T}), obtaining
\begin{align}\label{15eq: O=Q}
	\SO_{^{\pm}}^{_{-+}} (\upphi) = \sum_{c \Lt P/ T} \sum_{q \Lt P/c T} \frac { Q_q^{\scriptscriptstyle \pm } (c; \upphi) } {c q}  + O \big(T^{-A}\big), 
\end{align} 
where 
\begin{align}\label{14eq: Q(c)}
	Q^{\scriptscriptstyle \pm}_q (c; \upphi) = \sum_{p \leqslant P } \overline{V_q (p, \pm 1; c)} \frac {\log p} {2 \sqrt{p} \log T}  I^{_\pm}_{^{- +} }  \bigg( \frac {2\pi  {p}} {c q}   \bigg)  \hat{\upphi}  \bigg(\frac { \log p} {\log T} \bigg)  . 
\end{align}
In the asymptotic formula \eqref{12eq: sum of Vq},  we drop the condition $p \nmid c$ and change $\log p$ into $ \log p /\sqrt{p}$ by partial summation, and it follows that
\begin{align}\label{14eq: sum of Vq}
	 \sum_{p \leqslant x}     V_{q} ( p, \pm 1 ; c)\frac { \log p } {2 \sqrt{p}} = \frac {R (1; c) R_q  (\pm 1; c) } {\varphi (c)}   \sqrt{x}  + O ( c  (cx)^{\vepsilon} ) . 
\end{align} 
By \eqref{14eq: Q(c)} and \eqref{14eq: sum of Vq}, we get
\begin{align*}
	Q^{\scriptscriptstyle \pm}_q (c; \upphi) = - \frac 1 {\log T} \int_0^{P} \bigg\{  \frac {\overline{R (1; c) R_q  (\pm 1; c)} } {\varphi (c)}   \sqrt{x}  + O ( c  (cx)^{\vepsilon} ) \bigg\}\nd I^{_\pm}_{^{- +} }  \bigg( \frac {2\pi  {x}} {c q}   \bigg)  \hat{\upphi}  \bigg(\frac { \log x} {\log T} \bigg). 
\end{align*} 
On partial integration for the main-term integral and the change of variable $v = 2\pi x/ cq$, in view of \eqref{12eq: R(1;c)} and \eqref{12eq: Rq(1;c)}, we deduce that
\begin{align*}
	Q^{\scriptscriptstyle \pm}_q (c; \upphi) \Lt_{\upphi} T^{\vepsilon}  \int_0^{\frac {2\pi P} {cq}} \bigg(   \bigg(\frac {\sqrt{q}} {\sqrt{c v}} + \frac {c} {v} \bigg) |I^{_\pm}_{^{- +} }   ( v ) |  + c | I^{_{\pm \natural}}_{^{- +} }   (   v   )  |   \bigg) { \nd v} , 
\end{align*} 
\delete{the main term is equal to 
\begin{align*} 
  	\frac {\sqrt{ c q} R (1; c) R_q  (\pm 1; c)} { \sqrt{2\pi} \log T \cdot   \varphi (c) }   \int_0^{2\pi P/cq}  I^{_\pm}_{^{- +} }   ( v )  \hat{\upphi}  \bigg(\frac { \log (cq v/ 2\pi)} {\log T} \bigg) \frac { \nd v} {\sqrt{v} } ,  
\end{align*}
while the error term is bounded by 
\begin{align*}
	 c(c P)^{\vepsilon} \int_0^{2\pi P/cq}  \big(    | I^{_{\pm \natural}}_{^{- +} }   (   v   )  | +    | I^{_\pm}_{^{- +} }   (   v   ) / v  | \big) \nd v, 
\end{align*}}
where $ I^{_{\pm \natural}}_{^{- +} }  (   v ) $ is the (partial) derivative of  $ I^{_\pm}_{^{- +} }  (   v ) $ defined as in \eqref{13eq: derivative I(v,w)}. By Lemma \ref{lem: further analysis}, the domain of integration may be restricted to $|2T - v| \Lt \varPi^{\vepsilon} ( \varPi +T / \varPi )$, and it yields the bound: 
\begin{align}\label{15eq: Qq(c)}
	Q^{\scriptscriptstyle \pm}_q (c; \upphi) \Lt T^{\vepsilon} \bigg( \varPi + \frac {T} {\varPi} \bigg) \bigg( \frac {\sqrt{q T}} {\sqrt{c  } } +  \frac {c T} {\varPi} \bigg) . 
\end{align} 
Now we insert \eqref{15eq: Qq(c)} into \eqref{15eq: O=Q}, 
\begin{align*}
	\SO_{^{\pm}}^{_{-+}} (\upphi) & \Lt T^{\vepsilon} \bigg( \varPi + \frac {T} {\varPi} \bigg) \sum_{c \Lt P/ T} \sum_{q \Lt P/c T} \frac 1 {c q}  \bigg( \frac {\sqrt{q T}} {\sqrt{c  } } +   \frac {c T} {\varPi} \bigg) \\
	& \Lt T^{\vepsilon} \bigg( \varPi + \frac {T} {\varPi} \bigg) \sum_{c \Lt P/ T}    \bigg( \frac {\sqrt{P}} { {c^2  } } +   \frac {  T} {\varPi} \bigg) \\
	& \Lt T^{\vepsilon} \bigg( \varPi + \frac {T} {\varPi} \bigg)   \bigg(   {\sqrt{P}}   +   \frac {  P } {\varPi} \bigg). 
\end{align*} 
Finally, we verify $ \SO_{^{\pm}}^{_{-+}} (\upphi) = o (\varPi T) $ in various cases. Keep in mind that $\varPi = T^{\mu}$ and $P = T^{v}$.  For $1/2 < \mu < 1$,  $ \SO_{^{\pm}}^{_{-+}} (\upphi) = O (T^{\vepsilon} (\varPi \sqrt{P} + P)) $ and this is $o (\varPi T)$ as long as $v < 1 + \mu$.  For $1/3 < \mu \leqslant 1/ 2$, $ \SO_{^{\pm}}^{_{-+}} (\upphi) = O (T^{\vepsilon} (\varPi \sqrt{P} + P) T / \varPi^2)  $ and this is satisfactory for $v < 3 \mu$. 
 For $0 < \mu \leqslant 1/3$, $\SO_{^{\pm}}^{_{-+}} (\upphi)$ is negligible for $v < 1$ as the $c$-sum is void (of course, in this case \eqref{15eq: P(phi)} also follows from \eqref{15eq: P(phi), 0} in Lemma \eqref{lem: P limited}).  Moreover, in view of Remark \ref{rem: Ipm}, one may treat the other $ \SO_{^{\pm}}^{_{\pm\pm}} (\upphi)$ in a similar or easier manner.  

\section{Proof of the Non-vanishing Results on the Riemann Hypothesis}\label{sec: non-vnaishing, 2} 

Let us modify the argument  in \cite{ILS-LLZ} as follows.  
Put
\begin{align}\label{16eq: defn of pm}
	p^{\delta}_{m} (T, \varPi) =   \sum_{ f \in \SB_{\delta} } \omega_f \delta  (\mathop{\mathrm{ord}}_{ s = s_f} L (s, f) = m   ) \upvarphi_{T, \varPi} (t_f)  \Big/ \sum_{ f \in \SB_{\delta} } \omega_f \upvarphi_{T, \varPi} (t_f) ,  
\end{align}
where the $\delta $ in the sum is the Kronecker $\delta$-symbol (not to be confused with the parity $\delta$ in $\SB_{\delta}$). Clearly
\begin{align} \label{16eq: sum pm}
	\sum_{m=0}^{\infty}  p^{\delta}_{m} (T, \varPi) = 1. 
\end{align}
On the other hand, we choose as in  \cite[(1.42)]{ILS-LLZ} the  Fourier pair
\begin{align}
	\upphi (x) = \bigg(\frac {\sin (\pi v x)} {\pi v x} \bigg)^2 , \quad \hat{\upphi} (y) = \frac 1 {v} \bigg(1 - \frac {|y|} {v} \bigg), \qquad \text{($|y| < v$)},
\end{align}
for every $v < v (\mu)$ (see \eqref{15eq: v(mu)}). By the definition in \eqref{1eq: defn D(f)}, it follows from $ \upphi (0) = 1 $ and the non-negativity of $\upphi$ that 
\begin{align*}
	D  (f; \upphi; T) \geqslant \mathop{\mathrm{ord}}_{ s = s_f} L (s, f) . 
\end{align*} Also note that $$ \hat{\upphi} (0) = \frac 1 {v} .$$ Thus,  for $ 1/3 < \mu < 1$, we derive from \eqref{1eq: density}, \eqref{1eq: lim, extended}, and \eqref{1eq: Weyl, smooth} that 
\begin{align}\label{16eq: sum m pm}
	\sum_{m=1}^{\infty} m  p^{\delta}_{m} (T, \varPi) < \frac 1 {v (\mu)}  + \vepsilon, \qquad \varPi = T^{\mu}, 
\end{align}
for any $\vepsilon > 0$, provided that $T$ is sufficiently large. 
By subtracting \eqref{16eq: sum m pm} from \eqref{16eq: sum pm}, we obtain the lower bound
\begin{align}\label{16eq: lower bound}
	p^{\delta}_{0} (T, T^{\mu})  > 1 - \frac 1 {v (\mu)} -\vepsilon . 
\end{align}
Note that \eqref{16eq: lower bound} is non-trivial only if $ v (\mu) > 1 $, so we do need $ 1/3 < \mu < 1$ as above. 
In view of \eqref{15eq: v(mu)} and \eqref{16eq: defn of pm}, let us reinterpret \eqref{16eq: lower bound} as 
\begin{align}
	\sum_{f \in \SB_{\delta} : L (s_f, f) \neq 0} \omega_f \upvarphi_{T, T^{\mu}} (t_f)  > \bigg( \min \bigg\{ \frac {3\mu-1} {3\mu}, \frac {\mu } { \mu + 1} \bigg\} - \vepsilon \bigg) \sum_{f \in \SB_{\delta}  } \omega_f \upvarphi_{T, T^{\mu}} (t_f) . 
\end{align}
Similar to the argument in \S \ref{sec: conclusion, non-vanishing}, we may remove the smooth and harmonic weights by Lemma \ref{lem: unsmooth, 1} and the technique of  Kowalski--Michel \cite{KM-Analytic-Rank}. Consequently, we arrive at \eqref{1eq: main, long, RH} and \eqref{1eq: main, RH} in Theorems \ref{thm: non-vanishing, 3} and \ref{thm: non-vanishing, 4}.

\begin{appendices}

\section{Addendum: Non-vanishing of Central $L$-values on the Riemann Hypothesis}\label{sec: non-vnaishing, 1/2} \label{sec: addendum} 

In this addendum, we take the opportunity to improve Theorem A.3 in {\rm\cite{Qi-Liu-LLZ}} as follows, with the restriction $\mu > 1/3$ removed, and also to establish conditional non-vanishing results  by (1.40)--(1.45) in \cite{ILS-LLZ}  for the central $L$-values $L (1/2 , f)$ or $L' (1/2, f)$ in the even or odd case.  

\begin{theorem}\label{thm: density, L(1/2)}
Let the notation be as in Appendix in {\rm\cite{Qi-Liu-LLZ}}. 	 Assume the Riemann hypothesis for Dirichlet $L$-functions. Let  $ T, M > 1$ be such that $M = T^{\mu}$ with $0 < \mu < 1$.  Define 
\begin{align}
	v (\mu) = \left\{ \begin{aligned}
		& 1, & & \text{ if } 0 < \mu \leqslant 1/4,\\
		& 4 \mu, & & \text{ if } 1/4 < \mu \leqslant 1/3, \\
		 & 1 + \mu, & & \text{ if } 1/3 < \mu < 1.
	\end{aligned}\right.
\end{align}
Let $\upphi$ be an even Schwartz function with the support of $\hat{\upphi}$ in $(-v (\mu), v (\mu) )$.  Then 
\begin{align}
	\lim_{T \ra \infty} \SD_1 (H^{+} (1), \upphi; h_{T, M}) & = \int_{-\infty}^{\infty} \upphi (x) W  (\mathrm{SO}(\mathrm{even})) (x) \nd x,  \\
	\lim_{T \ra \infty} \SD_1 (H^{-} (1), \upphi; h_{T, M}) & = \int_{-\infty}^{\infty} \upphi (x) W  (\mathrm{SO}(\mathrm{odd})) (x) \nd x .
\end{align}
\end{theorem}

\begin{proof}
It only requires minor modifications to the proof of Theorem A.3  (indeed, Lemmas A.8 and A.9) in {\rm\cite{Qi-Liu-LLZ}}. 

As in (A.21) and (A.23) in {\rm\cite{Qi-Liu-LLZ}}, we need to analyze integrals of the form:  
\begin{align*}
H_{T, M}^{- \sharp} (x) =	\frac {2 M T } {\pi^2} \int_{-M^{\vepsilon}/M}^{M^{\vepsilon}/ M} 
	\widehat{k} (-Mr/\pi) e (   T r / \pi) \cos  ( 4 \pi x \sinh r )   \nd r, 
\end{align*}
\begin{align*}
H_{T, M}^{- \natural} (x) =	- \frac {8 M T } {\pi} \int_{-M^{\vepsilon}/M}^{M^{\vepsilon}/ M} 
	\widehat{k} (-Mr/\pi) \sinh r  \cdot e (     T r / \pi) \sin  ( 4 \pi x \sinh r )   \nd r, 
\end{align*} 
where $ \widehat{k} (- r/\pi) $ is a Schwartz function similar to the $g ( r )$ as in Lemma \ref{lem: Bessel}.  As in the second statement of Lemma A.8 in {\rm\cite{Qi-Liu-LLZ}}, under the condition $ \mu > 1 / 3 $, it is proven that $H_{T, M}^{- \sharp} (x) $ and $ H_{T, M}^{- \natural} (x) $ are negligibly small unless  $ |T-2\pi x| \leqslant M^{1+\vepsilon} $. The reader should compare these integrals with those in \eqref{13eq: I(v,w)} and \eqref{13eq: In(v,w)}.  Note that the derivatives of the phase functions here satisfy
\begin{align*}
 	T / \pi  -  2 x \cosh r = T /\pi - 2 x + O (xM^{\vepsilon} / M^2), \qquad T / \pi  +  2 x \cosh r > T/\pi. 
\end{align*}

For this theorem, it suffices to remove the condition $\mu > 1/3$ but weaken the range $ |T-2\pi x| \leqslant M^{1+\vepsilon}  $ into $ |T-2\pi x| \leqslant M^{\vepsilon} (M + T/M^2) $. To this end, we may apply  Lemma B.1 in \cite{Qi-GL(2)xG(2)-RS} with $S = 1$, $ P = 1/ M $, $Q = M^{\vepsilon} / M$, $Z = T / M^3 $, and $ R = M^{\vepsilon} ( M +T / M^2 ) $. This is essentially in the spirit of our Lemma \ref{lem: further analysis}. Now, for    $P  = T^{2v}$, it is easy to verify that the error bound $O (\sqrt{P} T^{\vepsilon})$   becomes $ O (\sqrt{P} (1+T/M^3) T^{\vepsilon}) $  (for $\sqrt{P} \Gt T$, otherwise the $c$-sum is void) in the first part of the proof of Lemma A.9  in \cite{Qi-Liu-LLZ}.
\end{proof}

Note that in the notation of this paper $\SB_0 = H^{+} (1)$, $\SB_1 = H^{-} (1)$, and $\varPi = M$. 

Next, we assume the Riemann hypothesis for every $L (s, f)$ with $ f \in H^{\pm} (1) $ and for all Dirichlet $L$-functions. The following results may be derived from Theorem \ref{thm: density, L(1/2)} by adapting the arguments from \S \S \ref{sec: conclusion, non-vanishing}, \ref{sec: non-vnaishing, 2}, and also \cite{ILS-LLZ}. The reader may compare Corollary \ref{cor: long}  with Corollary 1.7 in  \cite{ILS-LLZ}. 

\begin{coro}\label{cor: long}
	We have 
	\begin{align} \label{add: main, long}
	\liminf_{T \ra \infty} 	\frac {	\text{\rm \small \bf \#} \big\{ f \in H^{+} (1)  :    t_f   \leqslant   T , \, L (1/2 ,   f    ) \neq 0
			\big\}} { \text{\rm \small \bf \#} \big\{ f \in H^{+} (1)  :  t_f \leqslant   T 
			\big\} } >    \frac 9 {16}        ,  \\
		\liminf_{T \ra \infty} 	\frac {	\text{\rm \small \bf \#} \big\{ f \in H^{-} (1)  :    t_f   \leqslant   T , \, L ' (1/2 ,   f    ) \neq 0
			\big\}} { \text{\rm \small \bf \#} \big\{ f \in H^{-} (1)  :  t_f \leqslant   T 
			\big\} } >    \frac {15} {16}   . 
	\end{align}  
\end{coro}

\begin{coro}
We have 
	\begin{align} \label{add: main, short}
		\liminf_{T \ra \infty} 	\frac {	\text{\rm \small \bf \#} \big\{ f \in H^{+} (1)  :  |t_f - T| \leqslant   T^{\mu} , \, L (1/2 ,   f    ) \neq 0
			\big\}} { \text{\rm \small \bf \#} \big\{ f \in H^{+} (1)  : |t_f - T| \leqslant   T^{\mu} 
			\big\} } \geqslant   \frac 1 {4}        ,  \\
		\liminf_{T \ra \infty} 	\frac {	\text{\rm \small \bf \#} \big\{ f \in H^{-} (1)  : |t_f - T| \leqslant   T^{\mu} , \, L ' (1/2 ,   f    ) \neq 0
			\big\}} { \text{\rm \small \bf \#} \big\{ f \in H^{-} (1)  : |t_f - T| \leqslant   T^{\mu}
			\big\} } \geqslant    \frac {3} {4}  , 
	\end{align}  
if $ 0 < \mu \leqslant 1/4$,   and 
	\begin{align} \label{add: main, short, 2}
& 	\liminf_{T \ra \infty} 	\frac {	\text{\rm \small \bf \#} \big\{ f \in H^{+} (1)  :  |t_f - T| \leqslant   T^{\mu} , \, L (1/2 ,   f    ) \neq 0
		\big\}} { \text{\rm \small \bf \#} \big\{ f \in H^{+} (1)  : |t_f - T| \leqslant   T^{\mu} 
		\big\} }   >   1 - \frac 1 {v (\mu )} + \frac 1 {4 v (\mu )^2}   ,  \\
& \quad	\liminf_{T \ra \infty} 	\frac {	\text{\rm \small \bf \#} \big\{ f \in H^{-} (1)  : |t_f - T| \leqslant   T^{\mu} , \, L ' (1/2 ,   f    ) \neq 0
		\big\}} { \text{\rm \small \bf \#} \big\{ f \in H^{-} (1)  : |t_f - T| \leqslant   T^{\mu}
		\big\} }   >  1 - \frac 1 {4 v (\mu )^2}   , 
\end{align}  
if $1/4 < \mu < 1$, for 
\begin{align}
	v (\mu ) = \min \{ 4\mu, 1+\mu \}. 
\end{align}
\end{coro}

\end{appendices}

\def\cprime{$'$}

\end{document}